\documentclass[a4paper,10pt,leqno,twoside]{tobart}

\usepackage[english]{babel} \usepackage{inputenc, amsmath, amssymb, latexsym,
  epic, epsfig, graphicx, rotating, fancyhdr, amsthm, pifont, empheq}

\newcommand{\eqand}{\ensuremath{\quad \textrm{and} \quad}}

\newcommand{\follows}{\ensuremath{\Rightarrow}}

\newcommand{\ld}{\ensuremath{,\ldots,}}
\newcommand{\ssq}{\ensuremath{\subseteq}}
\newcommand{\smin}{\ensuremath{\setminus}}
\newcommand{\eps}{\ensuremath{\varepsilon}}

\newcommand{\inte}{\ensuremath{\mathrm{int}}}
\newcommand{\closure}{\ensuremath{\mathrm{cl}}}

\newcommand{\kreis}{\ensuremath{\mathbb{T}^{1}}}
\newcommand{\ntorus}[1][2]{\ensuremath{\mathbb{T}^{#1}}}

\newcommand{\sltr}{\ensuremath{\textrm{SL}(2,\mathbb{R})}}
\newcommand{\torus}{\ensuremath{\mathbb{T}^2}}

\newcommand{\nfolge}[1]{\ensuremath{(#1)_{n\in\mathbb{N}}}}

\newcommand{\jfolge}[1]{\ensuremath{(#1)_{j\in\mathbb{N}}}}

\newcommand{\twomatrix}[4]{\ensuremath{\left(\begin{array}{cc} #1 & #2 \\ #3 &
      #4 \end{array}\right)}}

\newcommand{\alphlist}{\begin{list}{(\alph{enumi})}{\usecounter{enumi}\setlength{\parsep}{2pt}
      \setlength{\itemsep}{1pt} \setlength{\topsep}{5pt}
      \setlength{\partopsep}{3pt}}}
\newcommand{\arablist}{\begin{list}{(\arabic{enumi})}{\usecounter{enumi}\setlength{\parsep}{2pt}
          \setlength{\itemsep}{1pt} \setlength{\topsep}{5pt}
          \setlength{\partopsep}{3pt}}}
\newcommand{\romanlist}{\begin{list}{(\roman{enumi})}{\usecounter{enumi}\setlength{\parsep}{2pt}
              \setlength{\itemsep}{1pt} \setlength{\topsep}{5pt}
              \setlength{\partopsep}{3pt}}}
\newcommand{\Romanlist}{\begin{list}{(\Roman{enumi})}{\usecounter{enumi}\setlength{\parsep}{2pt}
              \setlength{\itemsep}{1pt} \setlength{\topsep}{5pt}
              \setlength{\partopsep}{3pt}}}
\newcommand{\bulletlist}{\begin{list}{$\bullet$}{\setlength{\parsep}{2pt}
                \setlength{\itemsep}{1pt} \setlength{\topsep}{5pt}
                \setlength{\partopsep}{3pt}\setlength{\leftmargin}{15pt}}} 
\newcommand{\Alphlist}{\begin{list}{(\Alph{enumi})}{\usecounter{enumi}\setlength{\parsep}{2pt}
      \setlength{\itemsep}{1pt} \setlength{\topsep}{5pt}
      \setlength{\partopsep}{3pt}}}
 \newcommand{\listend}{\end{list}}

\newcommand{\T}{\ensuremath{\mathbb{T}}}

\newcommand{\N}{\ensuremath{\mathbb{N}}} 
\newcommand{\R}{\ensuremath{\mathbb{R}}}
\newcommand{\Z}{\ensuremath{\mathbb{Z}}}
\newcommand{\Q}{\ensuremath{\mathbb{Q}}}

\newcommand{\A}{\ensuremath{\mathbb{A}}}

\newcommand{\cA}{\mathcal{A}}
\newcommand{\cB}{\mathcal{B}}
\newcommand{\cC}{\mathcal{C}}
\newcommand{\cD}{\mathcal{D}}

\newcommand{\cF}{\mathcal{F}}

\newcommand{\cI}{\mathcal{I}}

\newcommand{\cL}{\mathcal{L}}
\newcommand{\cM}{\mathcal{M}}

\newcommand{\cP}{\mathcal{P}}

\newcommand{\cR}{\mathcal{R}}

\newcommand{\cU}{\mathcal{U}}
\newcommand{\cV}{\mathcal{V}}
\newcommand{\cW}{\mathcal{W}}
\newcommand{\cX}{\mathcal{X}}
\newcommand{\cY}{\mathcal{Y}}

\newcommand{\nLim}{\ensuremath{\lim_{n\rightarrow\infty}}}

\newcommand{\ntel}{\ensuremath{\frac{1}{n}}}

\newtheoremstyle{tobthm}{3pt}{3pt}{\itshape}{0pt}{\bfseries}{.}{0.5eM}{}
\theoremstyle{tobthm}

\newtheorem{definition}{Definition}[section]
\newtheorem{thm}[definition]{Theorem}

\newtheorem{lem}[definition]{Lemma}
\newtheorem{Lemma}[definition]{Lemma}

\newtheorem{prop}[definition]{Proposition}

\newtheorem{claim}[definition]{Claim}

\newtheoremstyle{tobrem}{3pt}{3pt}{\normalfont}{0pt}{\bfseries}{.}{0.5em}{}
\theoremstyle{tobrem}

\newtheorem{rem}[definition]{Remark}

\newtheorem{examples}[definition]{Examples}

\numberwithin{equation}{section} \numberwithin{figure}{section}

\title{\large \bf Abundance of mode-locking for quasiperiodically
  forced circle maps} \author{J.~Wang\thanks{TU Dresden, Department of
    Mathematics, 01062 Dresden, Germany. Email: {\tt
      jingwang018@gmail.com}} \ and
  T.~J\"ager\thanks{Friedrich-Schiller-University Jena, Institute of
    Mathematics, 07743 Jena, Germany. Email: {\tt
      Tobias.Oertel-Jaeger@tu-dresden.de}}}

\newcommand{\flthx}{\ensuremath{f_{\tau,\theta}(x)}}
\newcommand{\flth}{\ensuremath{f_{\tau,\theta}}}
\newcommand{\nofolge}[1]{\ensuremath{(#1)_{n\in\mathbb{N}_0}}}

\begin{document}

\setlength{\abovedisplayskip}{1.0ex}
\setlength{\abovedisplayshortskip}{0.8ex}

\setlength{\belowdisplayskip}{1.0ex}
\setlength{\belowdisplayshortskip}{0.8ex}

\maketitle

\abstract{We study the phenomenon of mode-locking in the context of
  quasiperiodically forced non-linear circle maps. As a main result, we show
  that under certain $\cC^1$-open condition on the geometry of twist parameter
  families of such systems, the closure of the union of mode-locking plateaus
  has positive measure. In particular, this implies the existence of infinitely
  many mode-locking plateaus (open Arnold tongues). The proof builds on
  multiscale analysis and parameter exclusion methods in the spirit of Benedicks
  and Carleson, which were previously developed for quasiperiodic \sltr-cocycles
  by Young and Bjerkl\"ov. The methods apply to a variety of examples, including
  a forced version of the classical Arnold circle map. }

\section{Introduction}

The paradigm example for the phenomenon of mode-locking in dynamical
systems is the Arnold circle map
\begin{equation}
  \label{eq:1}
  f_{\alpha,\tau} : \kreis\to\kreis \quad ,
  \quad x\mapsto x+\tau+\frac{\alpha}{2\pi}\sin(2\pi x) \ \bmod 1\ ,
\end{equation}
with non-linearity parameter $\alpha\in[0,1]$ and twist parameter
$\tau\in[0,1]$. If $F_{\alpha,\tau}:\R\to\R$ denotes the canonical
lift of $f_{\alpha,\tau}$, then its rotation number is given by
\begin{equation}
  \label{eq:2}
  \rho(F_{\alpha,\tau})\ =\ \nLim(F_{\alpha,\tau}^n(x)-x)/n\ .
\end{equation}
{\em Mode-locking} in this context refers to the fact that for
certain values of $\alpha$ and $\tau$ the mapping
$\tau'\mapsto\rho(F_{\alpha,\tau'})$ is locally constant in
$\tau'=\tau$. Maximal parameter intervals with constant rotation
number are called {\em mode-locking plateaus}. It is well-known that
for the Arnold circle map and similar parameter families
mode-locking is abundant. More precisely, for all $\alpha\in(0,1]$
the graph of $[0,1]\to[0,1],\ \tau\mapsto \rho(F_{\alpha,\tau})$ is
a {\em devils staircase}, that is, it is locally constant on an open
and dense subset while increasing from $0$ to $1$ over the unit
interval (e.g.~\cite[Chapter 11]{katok/hasselblatt:1997}). As a
basic model, this gives an understanding of mode-locking phenomena
occuring in a variety of real-world situations, including damped
pendula and electronic oscillators
\cite{Ding1986NonlinearOscillators}, heart-beat \cite{arnold:1991}
or paradoxical neural behaviour
\cite{perkeletal:1964,CoombesBressloff1999ModeLocking}.

Generalisations of these results to more complex situations and
higher dimension are certainly highly desirable. However, it turns
out that substantial difficulties have to be overcome in this
direction.  One particular example that demonstrates well this fact
is the so-called {\em Harper map}. It is the real-projective action
of a quasiperiodic \sltr-cocycle associated to the almost-Mathieu
operator, a discrete 1D Schr\"odinger operator with quasiperiodic
potential \cite{johnson/moser:1982,herman:1983,haro/puig:2006}. Due
to an intimate relation between orbits of the Harper map and formal
eigenfunctions of the almost-Mathieu operator, a fruitful blend of
methods from spectral theory, harmonic analysis and dynamical
systems can be used to analyse this model. Nevertheless, it has
taken decades before the existence of a devil's staircase had been
established for all parameters in several steps
\cite{bellissard/simon:1982,puig:2004,avila/jitomirskaya:2005}.

Here, our aim is to show abundance of mode-locking, in a slightly weaker sense
than above, for more general, non-linear {\em quasiperiodically forced (qpf)
  circle diffeomorphisms}. These are skew product diffeomorphisms of the form
\begin{equation}
  \label{eq:3}
  f : \ntorus\to\ntorus \quad , \quad (\theta,x)\mapsto(\theta+\omega,f_\theta(x))\ ,
\end{equation}
where $\omega\in\R\smin\Q$ and all fibre maps
$f_\theta:\kreis\to\kreis$ are circle diffeomorphisms. In addition,
we require $f$ to be homotopic to the identity and denote the class
of such maps by $\cF_\omega$, where $\omega$ refers to the rotation
number on the base. The Harper map mentioned above fits into this
setting, although the particular linear-projective structure makes
it quite special. In the genuinely non-linear case, much fewer
techniques are available for the investigation of such systems, and
the theory is far less developed in general.

Yet, there is one well-established method of choice for the analysis of qpf
circle diffeomorphisms in the hyperbolic regime -- characterised by
non-vanishing Lyapunov exponents -- which is multiscale analysis and parameter
exclusion in the spirit of Benedicks and Carleson
\cite{benedicks/carleson:1991}. In the above context, it was first developed by
Young \cite{young:1997} and Bjerkl\"ov \cite{bjerkloev:2005a,bjerkloev:2005} for
the linear-projective case and later adapted to non-linear systems in
\cite{jaeger:2009a,fuhrmann2013NonsmoothSaddleNodesI}.  Originally, this method
was used to show the non-uniform hyperbolicity of certain quasiperiodic
\sltr-cocycles \cite{young:1997,bjerkloev:2005a}, which corresponds to the
existence of strange non-chaotic attractors in the general case.

The principle goal of the present article is to develop this approach further,
and to show how it can be applied to the problem of mode-locking. The trick
which does this is a somewhat twisted argument. In a first step, parameter
exclusion is used to identity a large set of parameters for which the dynamics
are non-uniformly hyperbolic and minimal and no mode-locking occurs. These are
`good' parameters in the sense of the multiscale analysis scheme. In a second
step, the information obtained in this process is then used to show that a small
shift allows to change from any of these good parameters to a `bad' one,
previously excluded during the parameter elimination, at which the multiscale
analysis scheme terminates at a finite level and the system becomes uniformly
hyperbolic and mode-locked. As a result, this yields that a large set (in the
sense of positive Lebesgue measure) of parameters with non-uniformly hyperbolic
behaviour can be approximated with mode-locked parameters. \medskip

In order to formulate our main result, we denote by $\cP_\omega$ the
set of $\cC^1$-parameter families of qpf circle diffeomorphisms with
parameter $\tau\in\kreis$ and rotation number $\omega$ on the base,
that is
\begin{equation}\label{eq:4}
\cP_\omega\ =\ \left\{(f_\tau)_{\tau\in\kreis}\mid
f_\tau\in\cF_\omega \ \textrm{ for all }\tau \in \kreis \textrm{ and
} (\tau,\theta,x)\mapsto f_\tau(\theta,x) \textrm{ is }
\cC^1\right\} \ .
\end{equation}
Elements of $\cP_\omega$ will be denoted by $\hat f$, that is, $\hat
f=(f_\tau)_{\tau\in\kreis}$.  Any $f\in\cF_\omega$ lifts to a
diffeomorphism $F$ of $\A=\kreis\times\R$ of the form
$F(\theta,x)=(\theta+\omega,F_\theta(x))$, where each fibre map
$F_\theta:\R\to\R$ is a lift of the circle diffeomorphism
$f_\theta$. The {\em fibred rotation number} of $f$ is defined by
\begin{equation}
  \label{eq:5}
  \rho(f) \ = \ \nLim \left(F_\theta^n(x)-x\right)/n \ \bmod 1 \ ,
\end{equation}
where $F_\theta^n=F_{\theta+(n-1)\omega}\circ\ldots\circ F_\theta$.
This limit always exists and is independent of $\theta$ and $x$
\cite{johnson/moser:1982,herman:1983}. Given $\hat f\in\cP_\omega$,
we let
\begin{equation}
  \label{eq:6}
  \cM(\hat f) \ = \ \left\{\tau\in\kreis\mid \tau'\mapsto\rho(f_{\tau'})
 \textrm{ is locally constant in } \tau'=\tau\right\} \ .
\end{equation}
In other words, $\cM(\hat f)$ is the union of mode-locking plateaus of $\hat
f$. It is known that on the set $\cM(\hat f)$ the rotation number only takes
values in the module $\Q+\omega\Q$ \cite{bjerkloev/jaeger:2009}.

An $f$-invariant graph is the graph of a measurable function
$\varphi:\kreis\to\kreis$ which satisfies
\begin{equation}
  \label{eq:7}
  f_{\theta}(\varphi(\theta)) \ = \ \varphi(\theta+\omega) \ .
\end{equation}
Hereby, we will identify invariant graphs which coincide
Lebesgue-almost surely and implicitly speak of equivalence classes.
The (vertical) Lyapunov exponent of an invariant graph is given by
\begin{equation}
  \label{eq:8}
  \lambda(\varphi) \ = \ \int_{\kreis} \log \left|f'_\theta(\varphi(\theta))\right| \ d\theta \ .
\end{equation}
If an invariant graph is non-continuous, meaning that there is no
continuous representative in the equivalence class, and has negative
Lyapunov exponent, then it is called a {\em strange non-chaotic
  attractor (SNA)} \cite{grebogi/ott/pelikan/yorke:1984,keller:1996}.\smallskip

\begin{thm} \label{t.modelocking} Suppose $\omega$ is Diophantine and
  $\delta>0$. Then there exists a $\cC^1$-open set
  $\cU=\cU(\omega,\delta)\ssq\cP_\omega$ such that for all
  $(f_\tau)_{\tau\in\kreis}\in\cU$ there is a set $\Lambda^{\hat f}\ssq\kreis$
  of Lebesgue measure $\geq 1-\delta$ with the properties that \romanlist
\item for all $\tau\in\Lambda^{\hat f}$, the map $f_\tau$ has a
  (unique) SNA and the dynamics of $f_\tau$ are minimal;
\item $\Lambda^{\hat f}\ssq\partial\cM(\hat f)$.  \listend
\end{thm}
For a suitable $\cC^1$-open set $\cU\ssq\cP_\omega$, the existence of a set
$\Lambda^{\hat f}$ with property (i) has already been established in
\cite{Jaeger2012SNADiophantine}. Hence, the crucial point here is to show that
this set from \cite{Jaeger2012SNADiophantine} is contained in the boundary of
the union of mode-locking plateaus. The proof is based on the above-mentioned
multiscale analysis scheme from
\cite{young:1997,bjerkloev:2005a,jaeger:2009a}.\smallskip

We note that (ii) implies the existence of infinitely many open
mode-locking plateaus. Yet, at the same time these only take a very
small proportion of the parameter space, since the set $\Lambda^{\hat
  f}$ already accounts for measure $1-\delta$. This agrees with the
fact that an apparent `vanishing' of the mode-locking plateaus, coming
along with the occurrence of SNA, has been reported in numerical
studies
\cite{feudel/kurths/pikovsky:1995,glendinning/feudel/pikovsky/stark:2000}.
(However, it must be emphasized that it was left open by the authors
whether or not this observation is a numerical artifact.)  The
explanation prompted by Theorem~\ref{t.modelocking} is that the
majority of mode-locking plateaus persist, but simply become too small
to be detected numerically. The collapse of single plateaus has been
described in \cite{jaeger:2009a}, in contrast to the situation for the
unforced Arnold circle map.

The main aim of the present work is to show how multiscale analysis methods can
be applied to mode-locking problems in the non-linear setting. We believe that
it is possible to go further in this direction and to combine the presented
arguments with recent work by Bjerkl\"ov \cite{Bjerklov2014SchroedingerAHP}, who
extends the multiscale analysis of \cite{bjerkloev:2005a,bjerkloev:2005} to all
parameters, in order to prove the existence of a devil's staircase under similar
conditions as above. For the special case of quasiperiodic Schr\"odinger
cocycles with $\cC^2$-potential, such a result has been announced recently by
Wang and Zhang \cite{WangZhang2013Schroedinger,WangZhang2014CantorSpectrum}. In
this setting, however, results on mode-locking have also been established
earlier by different methods ({\em Ten Martini Problem},
  \cite{bellissard/simon:1982,puig:2004,avila/jitomirskaya:2005}).  \smallskip

The set $\cU$ in Theorem~\ref{t.modelocking} is characterised
explicitely by a number of $\cC^1$-estimates, which are stated in
Section~\ref{Quantitative}. This is important in the context of
applications, since it allows to check whether a given parameter
family belongs to the set $\cU$ or not. Thus, it can be shown that
the assertions of the theorem hold for specific examples.
\begin{examples}
  \alphlist
\item First, the above statement can easily be applied to parameter
  families of additively forced circle diffeomorphisms of the form
  \begin{equation}
    \label{eq:14}
   f_\tau(\theta,x) \ = \ (\theta+\omega,h(x)+\tau+V(\theta)) \ ,
  \end{equation}
  provided the circle diffeomorphism $h:\kreis\to\kreis$ and the
  forcing function $V:\kreis\to\kreis$ have suitable geometric
  properties. In order to give some explicit examples, suppose $p\geq
  2$, let $a_p(x)=\int_0^x 1/(1+|\xi|^p) \ d\xi$ and
\[
h_\alpha(x) \ = \ \pi\left(\frac{a_p(\alpha
    \iota(x))}{2a_p(\alpha/2)}\right) \
\]
where $\alpha\geq 1$, $\iota:\kreis\to (-1/2,1/2]$ is the lift of the identity
map on $\kreis$ and $\pi:(-1/2,1/2]\to\kreis$ is the canonical
projection. Further, assume that $V$ is such that for all but finitely many
$x\in\kreis$ the set $V^{-1}(\{x\})$ consists of exactly two points $\theta_1$
and $\theta_2$ and we have $V'(\theta_1)<0$ and $V'(\theta_2)>0$. Note that for
$p=2$ we have $a_p(x)=\arctan(x)$, and $V(\theta)=\cos(2\pi\theta)$ is a
possible choice of $V$. In this case, $f_\tau$ is the projective action of the
quasiperiodic $\sltr$-cocycle $(\theta,v)\mapsto(\theta+\omega,A(\theta)\cdot
v)$ with $A(\theta)=R_{V(\theta)+\tau}\cdot
\twomatrix{\alpha^{1/2}}{0}{0}{\alpha^{-1/2}}$, where $R_\vartheta$ is the
rotation matrix with angle $\vartheta$. Yet, for other values of $p$ no such
cocycle representation is available.

If $\omega$ is Diophantine and $\alpha$ is chosen sufficiently large, then the
parameter family $f_\tau(\theta,x)=(\theta+\omega,h_\alpha(x)+\tau+V(\theta)$
belongs to the set $\cU$ in Theorem~\ref{t.modelocking}, which will be
explicitely characterised in Section~\ref{Quantitative} below. The details are
easy to check, see \cite[Section 3.8]{jaeger:2009a} (compare also
\cite[Corollary 1.2]{Jaeger2012SNADiophantine}). Thus, in this case
$(f_\tau)_{\tau\in\kreis}$ satisfies the assertions of
Theorem~\ref{t.modelocking}.
\item The presented methods and results can also be applied to the
  quasiperiodically forced version of the Arnold circle map given in
  (\ref{eq:1}), with a suitable forcing function like
  $V_\beta(\theta)=\arctan(\beta\sin(2\pi\theta))\pi$ with large
  $\beta>0$. Strictly speaking, some modifications are needed to
  include this case. This results from the fact that the Arnold circle
  map does not show arbitrarily strong expansion, which we work with
  in our proofs below. However, this can be made up for by requiring a
  special shape of the forcing function, translating into a largeness
  assumption on $\beta$ above.

  The required modifications have been carried out in detail in
  \cite{jaeger:2009a,Jaeger2012SNADiophantine}, and it is on the level of an
  advanced exercise to implement them as well for our setting.  As a result, one
  obtains that in the parameter family
  $f_\tau(\theta,x)=(\theta+\omega,h_{\alpha,\tau}(x)+V_\beta(\theta))$ the
  boundary of $\cM(\hat f)$ has positive measure, provided $\omega$ is
  Diophantine, $\alpha\in(0,1)$ and $\beta>0$ is sufficiently large. We note
  that due to the different geometry, the measure of $\partial \cM(\hat f)$
  cannot be ensured to be close to $1$ in this case (compare
  \cite{jaeger:2009a}).
\item The most prominent example of a quasiperiodically forced system
  is probably the so-called Harper map, which is induced
  real-projective action of the quasiperiodic Schr\"odinger cocycle
  associated to the almost-Mathieu operator. It takes the form
  \[
  f_\tau(\theta,x) \ = \
  \left(\theta+\omega,\frac{1}{\pi}\arctan\left(\frac{-1}{\tan(\pi
        x)-\tau+\lambda\cos(2\pi\theta)}\right) \bmod 1 \right) \ .
  \]
  Again, a slight modification of our methods would allow to treat
  this example for large coupling parameters $\lambda>0$. However, as
  mentioned above, stronger results are available for this special
  case \cite{avila/jitomirskaya:2005,WangZhang2014CantorSpectrum}, so
  we refrain from providing any details.  \listend
\end{examples}
\medskip

\noindent{\bf Acknowledgements.} JW has been supported by a research
fellowship of the Alexander-Humboldt-Foundation. TJ has received
support of the German Research Council (Emmy-Noether grant Ja 1721/2-1
and Heisenberg-Fellowship Oe 538/7-1). The first ideas for this
project have been developed during the International conference on
Hamiltonian dynamcs, Nanjing 2011, and the authors would like to thank
the organisers for creating the opportunity and Hakan Eliasson for a
helpful discussion.

\section{Review of the multiscale analysis and outline of the
  proof} \label{Outline}

\subsection{Multiscale analysis of qpf circle maps.}

The aim of this section is to give an outline of the proof of
Theorem~\ref{t.modelocking}, in order to provide some guidance through
the technically rather involved later sections and to render these
more accessible. To that end, we first need to give a brief
description of the multiscale analysis established in
\cite{jaeger:2009a,Jaeger2012SNADiophantine}, on which our
construction builds. As mentioned, the main result in
\cite{Jaeger2012SNADiophantine} is the existence of a $\cC^1$-open set
$\cU\ssq\cP_\omega$ such that for all $\hat f\in\cU$ there is a set
$\Lambda^{\hat f}\ssq\kreis$ of measure $\geq 1-\delta$ which
satisfies assertion (i) in Theorem~\ref{t.modelocking}, that is, for
each $\tau\in\Lambda^{\hat f}$ the map $f_\tau$ has an SNA and minimal
dynamics. The proof hinges on the crucial fact that the existence of
an SNA follows from that of a {\em sink-source orbit}, that is, an
orbit that has positive Lyapunov exponent both forwards and backwards
in time \cite{jaeger:2009a}. In the context of Schr\"odinger
operators, this corresponds to the existence of an exponentially
decaying eigenfunction
\cite{haro/puig:2006,bjerkloev:2005a,jaeger:2006a}.

We will work with essentially the same sets $\cU$ and $\Lambda^{\hat
  f}$ as in \cite{Jaeger2012SNADiophantine}, and therefore need to
understand the geometric properties of the parameter families in
$\cU$ and the mechanism which leads to the existence of sink-source
orbits for parameters in $\Lambda^{\hat f}$.  A complete list of the
$\cC^1$-estimates characterising $\cU$ and precise versions of the
following statements will be given in the next section. Here, we try
to sketch an overall picture in order to give some intuition.
Roughly spoken, the geometry of parameter families in $\cU$ can be
described as follows.  We supress the dependence on the parameter
$\tau$, since the respective properties are supposed to be satisfied
uniformly over the parameter range.
\begin{itemize}
\item[(a)] There exists a small interval $E\ssq\kreis$ and a large
  interval $C\ssq\kreis$ such that for all $\theta\in\kreis$ the fibre maps
  $f_\theta$ are expanding on $E$ and contracting on $C$. This gives
  rise to an {\em expanding region} $\kreis\times E$ and a {\em
    contracting region} $\kreis\times C$.
\item[(b)] Both these regions are `almost invariant', in the following
  sense. There is a {\em critical region} $\cI_0\ssq \kreis$,
  consisting of two small intervals $I^1_0$ and $I^2_0$, such that for
  all $\theta\notin \cI_0$ the fibre map $f_\theta$ sends $\kreis\smin
  E$ into $C$. In other words, this means that $\pi_1\circ\left(
  f(\kreis\times E^c)\cap (\kreis \times C^c)\right) \ssq \cI_0+\omega$.
  Equivalently, the inverse $(f_\theta)^{-1}$ maps $\kreis\smin C$
  into $E$.
\item[(c)] If the parameter $\tau$ is varied, the two components $I^1_0$
  and $I^2_0$ move with respect to each other with some minimal speed.
\item[(d)] The images of $I_0^1\times C$ and $I_0^2\times C$ under $f$
  intersect $\kreis\times E$ {\em `transversely'} and qualitatively
  look as in Figure~\ref{Fig.SNA}(a).
\item[(e)] All fibre maps $f_\theta=f_{\tau,\theta}$ are {\em monotone} in the parameter
  $\tau$, that is, $\partial_\tau f_{\tau,\theta}(x)>0$ for all
  $(\tau,\theta,x)\in\T^3$. Here $\partial_\xi$ denotes the derivative with
  respect to a variable $\xi$.
\end{itemize}

Using these assumptions, the multiscale analysis in
\cite{jaeger:2009a,Jaeger2012SNADiophantine} concentrates on a
sequence of critical sets ${\cal C}_0 ,{\cal C}_1, {\cal C}_2,
\ldots$, which are defined recursively with respect to a
super-exponentially increasing sequence of integers $\nofolge{M_n}$
(time scales) in the following way.
\begin{eqnarray}
  {\cal A}_n & := &  \{(\theta,x) \mid \theta \in {\cal I}_n-(M_n-1)\omega,\ x
  \in C\} \ , \label{e.An} \\{\cal B}_n & := &  \{(\theta,x) \mid \theta \in {\cal
    I}_n+(M_n+1)\omega,\ x \in E\} \ ,\label{e.Bn}\\ {\cal C}_n & := &
  f_\tau^{M_n-1}({\cal A}_n) \cap f_\tau^{-M_n-1}({\cal B}_n),\label{e.Cn}  \\
  {\cal I}_{n+1} &  := &  \mathrm{int}(\pi_1({\cal C}_n)) \ .\label{e.In}
\end{eqnarray}
It is important to note that all the above sets and also the time
scales $M_n$ implicitely depend on the parameter $\tau$. We will
sometimes make this dependence explicit by writing
$\cI_n(\tau),\cC_n(\tau)$, ect. The projection $\cI_{n}$ of
$\cC_{n-1}$ will be called the {\em $n$-th critical region} of
$f_\tau$. In general, not much can be said about the critical sets and
critical regions.  However, it turns out that for a large set of
parameters $\Lambda^{\hat f}_n$ it is possible to obtain a very
precise control up to stage $n$ of the construction.  These sets
$\Lambda_n^{\hat f}$ are defined by the validity of the following
slow-recurrence conditions for the critical regions of $f_\tau$.
\begin{equation}\label{e.Xn}
\tag*{$(\mathcal{X})_n$} d(\mathcal{I}_j,\mathcal{X}_j)\ >\
3\eps_j\qquad \forall j=0,\ldots,n, \qquad\textrm{ and }
\end{equation}
\begin{equation}\label{e.Yn}
  \tag*{$(\mathcal{Y})_n$}
  d((\mathcal{I}_j-(M_j-1)\omega)\cup(\mathcal{I}_j+(M_j+1)\omega),\mathcal{Y}_{j-1})\ > \
  0 \quad  \forall j=1,\ldots,n \ ,
\end{equation}
where
\begin{eqnarray}
\mathcal{X}_n & = & \bigcup_{l=1}^{2K_nM_n}(\mathcal{I}_n+l\omega)\ , \\
\mathcal{Y}_n & =
&\bigcup_{j=0}^n\bigcup_{l=-M_j}^{M_j+2}(\mathcal{I}_j+l\omega) \ ,
\end{eqnarray}
with $\nfolge{K_n}$ an exponentially increasing sequence of integers
and $\nfolge{\eps_n}$ a sequence of positive numbers decreasing to
zero super-exponentially. We have
\begin{equation}\label{e.Lambda_n}
\Lambda^{\hat f}_n \  = \ \left\{\tau\in\kreis\mid
\textrm{conditions  \ref{e.Xn} and  \ref{e.Yn} are satisfied for the
map } f_\tau \right\} \ .
\end{equation}

The conditions \ref{e.Xn} and \ref{e.Yn} play a central role in the
construction (as in previous work in
\cite{bjerkloev:2005a,bjerkloev:2007a,jaeger:2009a,Jaeger2012SNADiophantine}).
The reason is that if \ref{e.Xn} and \ref{e.Yn} hold, then a number
of rather straightforward and mainly combinatorial arguments allow
to establish the following facts concerning the geometry of the
first $n+1$ critical sets and regions.
\begin{itemize}
\item[(i)] The critical sets are nested and non-empty, that is,
  $\cC_1\supseteq \ld \supseteq \cC_{n+1}\neq\emptyset$.
\item[(ii)] For all $j=1\ld n+1$ the critical region $\cI_j$ consists
  of exactly two intervals $I^1_j$ and $I^2_j$, each of which has
  length $\leq \eps_j$.
\item[(iii)] If we denote by
  $\cA_j^\iota=(I^\iota_j-(M_j-1)\omega)\times C$ and
  $\cB_j^\iota=(I^\iota_j+(M_j+1)\omega)\times E$ with $\iota=1,2$ the two
  connected components of $\cA_j$ and $\cB_j$, then the intersections
  $f^{M_j}(\cA_j^\iota)\cap f^{-M_j}(\cB_j^\iota)$ are `transversal'
  and qualitatively look as in Figure~\ref{Fig.SNA}(a), but the size of the involved
  sets decreases super-exponentially.
\item[(iv)] If the parameter $\tau$ is varied, then the two components
  $I^1_j$ and $I^2_j$ move relative to each other with a certain minimal
  speed.
\item[(v)] For any starting point $(\theta,x)\in \closure(\cC_n)$, the first
  $M_n$ forwards iterates remain in the expanding region `most of the
  time', whereas the first $M_n$ backwards iterates mostly remain in
  the contracting region.
\end{itemize}
Based on the above statements, the existence of sink-source orbits
can be established rather easily. Since all the $\Lambda_n^{\hat f}$
are large, the same is true for the intersection $\Lambda^{\hat
  f}=\bigcap_{n\in\N} \Lambda^{\hat f}_n$. Given $\tau\in\Lambda^{\hat
  f}$, the intersection $\cC=\bigcap_{n\in\N} \closure(\cC_n)$ is
non-empty due to (i), and it follows from (v) that any orbit
starting in $\cC$ is a sink-source orbit.

The crucial issue in the above statements is the qualitative
description of the geometry of the intersections
$f^{M_j}(\cA_j^\iota)\cap f^{-M_j}(\cB_j^\iota)$ in (iii). For the
first stage of the construction, this is quite plausible from the
above assumptions (a)--(e). If $M_0$ is chosen such that
$\cI_0+k\omega\cap \cI_0=\emptyset$ for all $k=-M_0+1\ld -1$, then
due to (b) the iterates $f^k(\cA_0^\iota)$ of $\cA_0^\iota$ all
remain in the contracting region $\kreis\times C$. Consequently, the
image $f^{M_0-1}(\cA_0^\iota)$ is a `strip' contained in
$I^\iota_0\times C$, which is very thin and more or less horizontal
due to the contraction insides $\kreis\times C$. A more precise
version of condition (d) then ensures that the next image
$f^{M_0}(\cA_0^\iota)$ is a thin strip with more or less uniform
slope, slanted either upwards or downwards. A similar argument
yields that the preimage $f^{-M_0}(\cB_0^\iota)$ is a very thin
horizontal strip, and the two sets intersect as depicted in
Figure~\ref{Fig.SNA}(a). The main issue in
\cite{jaeger:2009a,Jaeger2012SNADiophantine} is to ensure that for
most parameters, this qualitative picture remains valid on all
levels of the construction. This is achieved by showing that the
iterates $f^k(\cA^\iota_n)$ with $k=1\ld M_n-1$ remain in
$\kreis\times C$ at least most of the times, even if they may visit
the critical parts $\cI_0\times\kreis$ of the phase space and thus
leave the contracting region for short periods. We refer to
\cite[Section 4.1]{Jaeger2012SNADiophantine} for a more detailed
description of these ideas.


\subsection{Outline of the proof.} \label{ProofOutline}

The proof of Theorem~\ref{t.modelocking} directly builds upon this
multiscale analysis. However, the task is now quite different. Since
the existence of the set $\Lambda^{\hat f}$ of `good parameters'
with measure $\geq 1-\delta$ has already been established in
\cite{Jaeger2012SNADiophantine}, we may assume a priori that this
set exists, satisfies assertion (i) of Theorem~\ref{t.modelocking}
and moreover the recurrence conditions \ref{e.Xn} and \ref{e.Yn}
hold for all $\tau_0\in\Lambda^{\hat f}$. The aim is then to prove
that an arbitrarily small perturbation of $\tau_0$ allows to find a
nearby parameter $\tau$ for which $f_\tau$ displays mode-locking.
The crucial observation in this context is the fact that if
$\cC_n=\emptyset$ for some $n\in\N$, then $f_\tau$ has an attracting
continuous invariant curve and consequently its rotation number is
mode-locked. This is stated in
Proposition~\ref{p.modelocking_criterion} below.  Hence, what we
need to show is that an arbitrarily small shift of a parameter
$\tau_0\in\Lambda^{\hat f}$ allows to render the intersection
$\cC_n$ empty for some $n\in\N$, while at the same time keeping the
slow-recurrence conditions $(\cX)_{n-1}$ and \ref{e.Yn}.

  \begin{figure}[t]

    \begin{center}
       \includegraphics[width=0.9\linewidth]{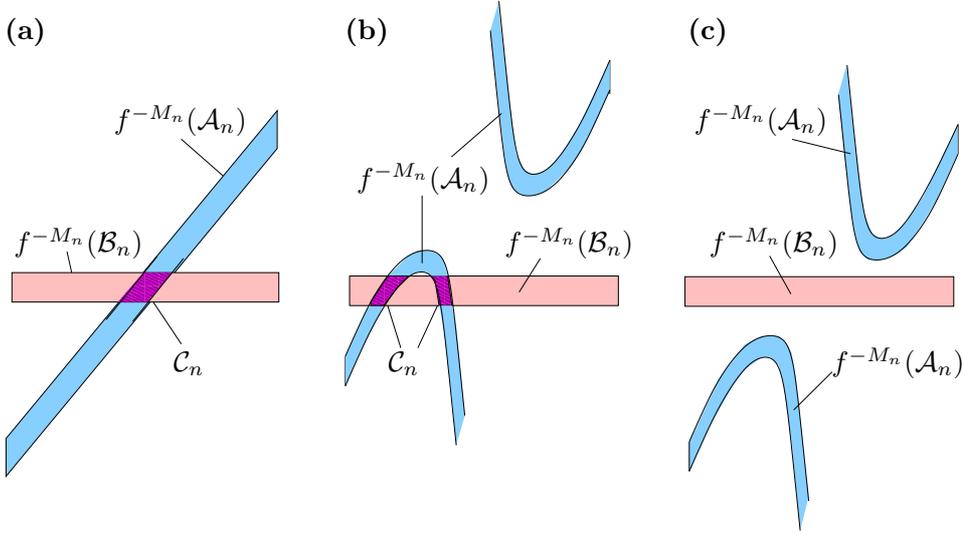}
     \end{center}
     \caption{\small The geometry of the critical sets $\cC_{n+1}$ in
       the multiscale analysis: (a) in the standard setting and (b)
       and (c) in the case of fast returns. Note that the two `hooks'
       of $f^{M_n}(\cA^1_n)$ are connected to each other as the set
       wraps once around the torus, but this is not depicted.
       \label{Fig.SNA}}
\end{figure}

In order to achieve this goal, we first perturb the parameter $\tau_0$
in such a way that the slow recurrence conditions $(\cX)_{n-1}$ and
$(\cY)_n$ still hold, but there is a {\em fast return} of $\cI_n$ to
itself.  More precisely, the control on the parameter-dependence of
the critical sets obtained in \cite{Jaeger2012SNADiophantine} is used
to shift $\tau_0$ in such a way that $I^1_n+k\omega\cap
I^2_n\neq\emptyset$ for some relatively small $k\geq 0$. For the
second component of $\cC_n$, on which we concentrate now, this implies
that when $\cA_n^2=(I^2_n-(M_n-1)\omega)\times C$ is iterated
forwards, it passes through the critical region $I^1_n\times\kreis$
before it approaches $I^2_n\times\kreis$ to intersect with the
$M_n$-th preimage of $\cB^2_n=(I^2_n+(M_n+1)\omega)$.

This results in a drastic change in the geometry of the resulting
image $f^{M_n}(\cA^2_n)$, and qualitatively the situation then looks
as in Figure~\ref{Fig.SNA}(b). The set $f^{M_n}(\cA^2_n)$ now has two
`hooks', and the vertical extension of the gap between these hooks is
greater than that of $f^{-M_n}(\cB^2_n)$. A more detailed explanation
for this behaviour is difficult to give at this stage, but will be
provided in Section~\ref{GeometricEstimates} (see
Figures~\ref{f.stragegy} and \ref{f.stragegy2}). Moreover, when the
parameter $\tau$ is varied further, the hooks move both horizontally
and, more importantly, also vertically with $\tau$, whereas the set
$f^{-M_n}(\cB^2_n)$ remains more or less stable. As a consequence, for
some parameters $\tau$ the involved sets have to reach a position
where their intersection remains empty, as shown in
Figure~\ref{Fig.SNA}(c).  A similar picture holds simultaneously for
the first component $f^{M_n}(\cA_n^1)\cap f^{-M_n}(\cB_n^1)$, thus
showing that $\cC_{n}=\emptyset$ for some $\tau$ close to $\tau_0$.
As mentioned above, this will allow to complete the proof of
Theorem~\ref{t.modelocking} via
Proposition~\ref{p.modelocking_criterion}.

For the rigorous implementation of this proof, the major task will
be to describe the geometry and parameter-dependence of the set
$f^{M_n}(\cA_n)$. Instead of trying to define precisely what it
means to be `hook-shaped', we show that $f^{M_n}(\cA_n^\iota)\cap
(\kreis\times E)$ is contained in the disjoint union of certain
polygons $\cL$ and $\cR$. We provide quantitative estimates on the
shape and position of these sets which imply that the preimage
$f^{-M_n}(\cB_n^\iota)$, which is a thin and more or less horizontal
strip contained in $\kreis\times E$, cannot intersect both of them
at the same time. Moreover, we show that by moving $\tau$ it is
possible to force an intersection with $\cL$ at one parameter near
$\tau_0$ and with $\cR$ at another one, which implies that for some
intermediate parameter there is no intersection with either of them.

The precise quantitative version of our main result, including the
explicit characterisation of the set $\cU$ in terms of
$\cC^1$-estimates, is given in the next section.
Section~\ref{Preliminaries} collects some further information and
statements on the multiscale analysis scheme from
\cite{jaeger:2009a,Jaeger2012SNADiophantine,fuhrmann2013NonsmoothSaddleNodesI}.
In Section~\ref{GeometricEstimates}, we state the properties and
quantitative estimates for the polygons $\cL$ and $\cR$ containing
$f^{M_n}(\cA_n^\iota)$ and on $f^{-M_n}(\cB_n^\iota)$ and show how
these statements imply the main result. The proofs of these estimates
are then given in Section~\ref{MainProof}.

\section{Quantitative version of the main result} \label{Quantitative}

We first state the precise conditions on the geometry of the
considered parameters families, which were only circumscribed in the
previous section.\medskip

\noindent \textit{ I. Diophantine condition.} We say
$\omega\in\kreis$ satisfies the Diophantine condition with constants
$\gamma$ and $\nu$ if
\begin{equation}
  \label{eq:Diophantine} 
  d(n\omega,0) \ > \ \gamma\cdot |n|^{-\nu} \quad \forall n\in\Z\setminus\{0\} \ .
\end{equation}
By $\cD(\gamma,\nu)$, we denote the set of $\omega\in\kreis$ which
satisfy \eqref{eq:Diophantine}.\smallskip

\noindent \textit{ II. Critical regions.} Let $E=[e^-,e^+]$ and
$C=[c^-,c^+]$ be two non-empty, compact and disjoint subintervals of
$\T^1$. We assume that for all $\tau\in\kreis$ there exists a set
${\cal I}_0(\tau) \subseteq \T^1$ which is the union of two disjoint
open intervals $I_0^1(\tau), I_0^2(\tau)$ and satisfies
\begin{equation}  \label{eq:Cinvariance} \tag{${\cal A}1$}
  \flth(\mbox{cl}(\T^1 \setminus E)) \ \subseteq \ \mbox{int}(C) \ \ \ \ \ \forall
  \theta \notin {\cal I}_0(\tau) \ .
\end{equation}
Note that this implies
\begin{equation} \label{eq:Einvariance} \tag{${\cal A}1'$}
  \flth^{-1}(\mbox{cl}(\T^1 \setminus C)) \ \subseteq \ \mbox{int}(E) \ \ \ \ \ \forall
  \theta \notin {\cal I}_0(\tau)+\omega \ .
\end{equation} 

\noindent \textit{III. Bounds on the derivatives.} Concerning the
derivatives of the fibre maps $\flth$, we assume that for some
$\alpha>1$ and $p\geq 2$ we have
\begin{equation} \tag{${\cal A}2$}
  \label{eq:bounds1}  \alpha^{-p} \ < \ \partial_x\flthx \ < \ \alpha^p \hspace{2eM} \quad
  \forall (\theta,x) \in \torus \ ;
\end{equation}
\begin{equation} \label{eq:bounds2} \hspace{3eM}  \partial_x\flthx \ > \
\alpha^{2/p}
  \hspace{4.1eM} \quad \forall (\theta,x) \in \T^1 \times  \tag{${\cal A}3$}
  E \ ; \end{equation}
\begin{equation} \label{eq:bounds3} \hspace{3eM} \partial_x\flthx \ < \
  \alpha^{-2/p}\hspace{4.1eM} \quad \forall (\theta,x) \in \T^1 \times C \ .
  \tag{${\cal A}4$}
 \end{equation}
 Further, we fix $S>0$ such that
\begin{equation} \label{eq:bounddth}  \tag{${\cal A}5$}
   |\partial_\theta \flthx| \ < \ S \ \ \ \ \ \forall (\theta,x) \in
  \torus \ .
\end{equation}

\noindent \textit{IV. Transversal Intersections.} The following
condition ensures that the image of $I_0^\iota(\tau) \times C$
crosses $(I^\iota_0(\tau)+\omega) \times E$ exactly once and not
several times.
\begin{equation} \label{eq:crossing}  \tag{${\cal A}6$}
  \begin{array}{l} \exists!\theta_\iota^1 \in I_0^\iota(\tau) \textrm{ with  }
    f_{\tau,\theta_\iota^1}(c^+) = e^- \textrm{ \ and } \\ \exists! \theta_\iota^2 \in
    I_0^\iota(\tau) \textrm{ with } f_{\tau,\theta_\iota^2}(c^-) = e^+ \ . \end{array}
\end{equation}
The slope of $f(I^\iota_0(\tau)\times C)$ is controlled by
\begin{equation}
\left\{ \begin{array}{ll} &\partial_\theta \flthx \ < \ -s \ \ \ \ \
\forall (\theta,x) \in I_0^1(\tau)
  \times \T^1\\
&\partial_\theta \flthx \ > \quad \, s \ \ \ \ \ \forall (\theta,x)
\in I_0^2(\tau)
  \times \T^1
  \end{array}\right.\ \ , \label{eq:s} \tag{${\cal A}7$}
\end{equation}
where $s$ is a constant with $0 < s < S$. Note that thus
$f(I^\iota_0(\tau)\times C)$ crosses $(I^\iota_0(\tau)+\omega)\times E$
`downwards' if $\iota=1$ and `upwards', as in
Figure~\ref{Fig.SNA}(a), if $\iota=2$.  \smallskip

\noindent \textit{ V. Dependence on $\tau$.} First, we assume that
$\flthx$ is monotonically increasing with respect to $\tau$,  and we
 fix upper and lower bounds $L,\ell>0$ on $\partial_\tau \flthx$, that is,
\begin{equation} \label{eq:bounddlambda} \tag{${\cal A}8$}
\ell< \partial_\tau \flthx \ < \ L\ \ \ \ \ \forall (\theta,x) \in
\torus \ .
\end{equation}
Writing $I^\iota_0(\tau) = (a^\iota_0(\tau),b^\iota_0(\tau))$ for
$\iota=1,2$, we further assume that the functions
$a^\iota_0,b^\iota_0$ are continuously differentiable with respect to
$\tau$. Then we assume
\begin{equation}
  \label{e.I_0-derivative} \tag{${\cal A}9$}
  \inf_{\tau\in\kreis}\left(\min\{\partial_\tau a^1_0(\tau),\partial_\tau b^1_0(\tau)\} -
  \max\{\partial_\tau a^2_0(\tau),\partial_\tau b^2_0(\tau)\}\right) \ > \ \ell/S \ . \\
\end{equation}
This ensures that the two components of $\cI_0$ `move relative to each
other' with minimal speed $\ell/S$. Finally, by increasing $L$ further
if necessary, we can assume that
\begin{equation}\label{eq:boundinterval}\tag{$\mathcal{A}10$}
  \sup_{\tau\in\kreis}\max\{|\partial_\tau a^1_0(\tau)|,|\partial_\tau b^1_0(\tau)|,
  |\partial_\tau a^2_0(\tau)|,|\partial_\tau b^2_0(\tau)| \} \ < \ 2L/s \ .
\end{equation}

Given $A\ssq\kreis$, we denote by $|A|$ the Lebesgue measure of $A$.
In particular, if $A$ is an interval, then $|A|$ is simply its
length. The quantitative version of Theorem~\ref{t.modelocking},
with an explicit characterisation of the set $\cU$, now reads as
follows.
\begin{thm} \label{t.modelocking_quantitative} Let $\omega\in\cD(\gamma,\nu)$,
  $\delta>0$ and suppose $\hat f\in \cP_\omega$ satisfies the
  conditions \eqref{eq:Cinvariance}--~\eqref{eq:boundinterval} above.
  Let
  $\eps_0=\sup_{\tau\in\kreis}\max\left\{|I^1_0(\tau)|,|I^2_0(\tau)|\right\}$.

  Then there exist contants
  $\alpha_*=\alpha_*(\delta,\gamma,\nu,p,S,s,\ell,L)>0$ and
  $\eps_*=\eps_*(\delta,\gamma,\nu,p,S,s,\ell,L)$ such that if
  $\alpha>\alpha_*$ and $\eps_0<\eps_*$, then there exists a set
  $\Lambda^{\hat f}\ssq\kreis$ of measure at least $1-\delta$ with the
  property that \romanlist
\item for all $\tau\in\Lambda^{\hat f}$, the map $f_\tau$ has a
  (unique) SNA and the dynamics of $f_\tau$ are minimal;
\item $\Lambda^{\hat f}\ssq\partial\cM(\hat f)$.  \listend
\end{thm}
Note that since the above conditions
\eqref{eq:Cinvariance}--\eqref{eq:boundinterval} are all
$\cC^1$-open, this directly implies Theorem~\ref{t.modelocking}.

\section{Preliminaries on the multiscale analysis}\label{Preliminaries}

As mentioned before, the existence of SNA and the minimality of the dynamics in
Theorems~\ref{t.modelocking} and \ref{t.modelocking_quantitative} are already
contained in \cite{jaeger:2009a,Jaeger2012SNADiophantine}.  However, in order to
build on these results, we need restate them in a precise way and provide some
additional quantitative information. In particular, this concerns the slow
recurrence conditions \ref{e.Xn} and \ref{e.Yn}, which are replaced by the
following stronger versions.
\begin{equation}\label{e.Xn'}
    \tag*{$(\mathcal{X'})_n$}
    d(\mathcal{I}_j,\mathcal{X}_j)\ >\ 9\varepsilon_j\qquad \forall
    j=0,\ldots,n,
  \end{equation}
  \begin{equation}\label{e.Yn'}
    \tag*{$(\mathcal{Y}')_n$}
    d((\mathcal{I}_j-(M_j-1)\omega)\cup(\mathcal{I}_j+(M_j+1)\omega),\mathcal{Y}_{j-1})\ > \
    2\eps_{j-1} \quad  \forall j=1,\ldots,n \ .
  \end{equation}
  With these notions, we can restate \cite[Theorem
  3.1]{Jaeger2012SNADiophantine} as follows. The information on the
  sequences $\jfolge{K_j},\jfolge{M_j}$ and \jfolge{\eps_j} is taken
  from the proof of this theorem.
  \begin{thm}[\cite{Jaeger2012SNADiophantine}]
    \label{t.Main_qualitative} Let $\omega\in\cD(\gamma,\nu)$,
    $\delta>0$ and suppose $\hat f\in \cP_\omega$ satisfies the
    conditions \eqref{eq:Cinvariance}--\eqref{eq:boundinterval} above.
    Let
    $\eps_0=\sup_{\tau\in\kreis}\max\{|I^1_0(\tau)|,|I^2_0(\tau)|\}$.
    Then there exist constants $\alpha_*'$ and $\eps_*'$, both
    depending on the constants $\delta,\gamma,\nu,p,S,s,\ell,L$ above,
    with the property that if $\alpha>\alpha_*'$ and $\eps_0<\eps_*'$,
    then there exists a set $\Lambda^{\hat f}\ssq\kreis$ of measure at
    least $1-\delta$ such that for all $\tau\in\Lambda^{\hat f}$ the
    map $f_\tau$ has an SNA and minimal dynamics.

  Further, for each $\tau\in\Lambda^{\hat f}$ there exist sequences
  $\jfolge{K_j},\jfolge{M_j}$ and \jfolge{\eps_j} such that for all
  $n\in\N$ the critical regions $\cI_n$ defined in
  \eqref{e.An}--~\eqref{e.In} satisfy the slow-recurrence assumptions
  \ref{e.Xn'} and \ref{e.Yn'} and in addition
   \begin{equation}
     \label{eq:10}
     \max\{|I^1_n|,|I^2_n|\} \ \leq \ \eps_n \ .
   \end{equation}
   Moreover, the above sequences can be chosen such that $M_0=3$,
   $K_j=2^{j+t+2}$ for some $t\geq 4$ which satisfies
   \begin{equation}\label{e.t}
     2^{-t} \ \leq \ \log\left(\frac{p^2+2}{p^2+1}\right) \ ,
   \end{equation}
    and for all $j\in\N_0$ we have
  \begin{eqnarray}
    \label{e.M_j}
      M_{j+1} & \in & \left[\alpha^{M_j/2pq},2\alpha^{M_j/pq}\right] \ , \\
      \eps_{j+1} & \in & [2\alpha^{-M_j/p}/s,2\alpha^{-M_j/2p}/s] \ , \label{e.eps_j}
  \end{eqnarray}
  where $q=\max\{8,4\nu\}$.
\end{thm}
\begin{rem}
  We note that there are actually two small modifications in
  Theorem~\ref{t.Main_qualitative} in comparison to \cite[Theorem
  3.1]{Jaeger2012SNADiophantine}. 

  The first is just the correction of an unfortunate typo. In the
  statement of \ref{e.Yn'} on \cite[Page
  1488]{Jaeger2012SNADiophantine}, the given lower bound is $2\eps_j$
  instead of $2\eps_{j-1}$.  However, it can be seen from estimate
  (4.21) in \cite[Lemma 4.7]{Jaeger2012SNADiophantine} and its use in
  the proof of \cite[Lemma 4.9]{Jaeger2012SNADiophantine} that all the
  respective statements hold with a lower bound of $2\eps_{j-1}$.

  The second modification concerns the definition of $\cY_n$ in
  \ref{e.Yn}, where the index $l$ in the union on the right runs from
  $-M_j$ to $M_j+2$, instead of only from $-M_j+1$ to $M_j+1$ as in
  the respective definition in \cite{Jaeger2012SNADiophantine}. This
  is an adaption that we need to make for technical reasons. However,
  this difference does not have any influence on the proofs in
  \cite{Jaeger2012SNADiophantine}, which go through in literally the
  same way, so that the result remains valid in the above form.
\end{rem}

The statement of Theorem~\ref{t.Main_qualitative} provides the basis for our
further analysis.  In addition we will need a number of technical
lemmas which allow to control the behaviour of orbits of finite length on
time-scales corresponding to the slow-recurrence conditions \ref{e.Xn} and
\ref{e.Yn}. The philosophy of these statements is the following. Suppose
$(\theta_0,x_0)\in\kreis\times C$ and let
$(\theta_n,x_n)=f^n_\tau(\theta_0,x_0)$. Then the almost invariance of the
contracting region, given by (\ref{eq:Cinvariance}), implies that $x_n\in C$ as
long as $\theta_j\notin \cI_0$ for all $j=0\ld n-1$. Thus, an orbit that starts
in the contracting region will stay there as long as its $\theta$-coordinate
stays away from the critical region $\cI_0$. The key observation on which the
whole multiscale analysis hinges is the fact that even for longer orbits, whose
first coordinates do visit the critical regions, a similar statement
nevertheless holds at least `most of the times'. In order to make this precise,
let
\begin{equation} \label{e.Z_nV_n} \cV^-_n \ = \
  \bigcup_{j=0}^{n}\bigcup_{l=-M_j+2}^0(\mathcal{I}_j+l\omega) \eqand
  \cW^+_n:=\bigcup_{j=0}^{n}\bigcup_{l=1}^{M_j+1}(\mathcal{I}_j+l\omega)
\end{equation}
Then we have
\begin{Lemma}[{\cite[Lemma 4.4]{fuhrmann2013NonsmoothSaddleNodesI}},
  Forwards Iteration]
  \label{lem-ess} Suppose $f_\tau$ satisfies (\ref{eq:Cinvariance})
  and $(\mathcal{Y})_{n-1}$ holds.  Let $\mathcal{L}\geq 0$ be the
  first integer such that $\theta_{\mathcal{L}}\in \mathcal{I}_n$.
  Then
  \begin{equation}\tag*{$(\mathcal{B}1)_n$}\label{con-b1}
\left\{
\begin{array}{lll}
& \theta_0\notin \cV^-_{n-1}\\
& x_0\notin \inte(E)
\end{array}
\right.
\end{equation}
implies that
\begin{equation}\tag*{$(\mathcal{C}1)_n$}
  \theta_m\notin\cW^+_{n-1}\ \follows \  x_m\in \inte(C) \quad \forall
  m=1,\ldots,\mathcal{L} \ .
\end{equation}
\end{Lemma}
We note that in \cite{fuhrmann2013NonsmoothSaddleNodesI} the lemma is
stated under the additional assumption that $(\cX)_{n-1}$ holds as
well, but this is actually not needed and is not used in the proof.
The same applies to Lemma~\ref{lem-ess_back} below.

It can be seen from (\ref{eq:10})--(\ref{e.eps_j}) that for large
$\alpha$ the exceptional sets $\cV^-_n$ and $\cW^+_n$ are very small.
Hence, an orbit starting in $(\kreis\smin\cV^-_n)\times C$ typically
remains trapped in the contracting region most of the time, until it
enters $\cI_{n+1}\times C$. A similar statement holds for the
backwards iteration.  Let
\begin{equation}
\label{e.Z_n_W_n} \cV_{n}^+ \ = \
\bigcup_{j=0}^{n}\bigcup_{l=1}^{M_j}(\mathcal{I}_j+l\omega) \eqand
\cW^-_n \ = \
\bigcup_{j=0}^{n}\bigcup_{l=-M_j+1}^0(\mathcal{I}_j+l\omega) \ .
\end{equation}
\begin{Lemma}[{\cite[Lemma 4.4]{fuhrmann2013NonsmoothSaddleNodesI}},
  Backwards Iteration] \label{lem-ess_back}Suppose $f_\tau$ satisfies
  (\ref{eq:Cinvariance}) and $(\mathcal{Y})_{n-1}$ holds. Let
  $\mathcal{R}\geq 0$ be the first integer such that
  $\theta_{-\mathcal{R}}\in \mathcal{I}_n+\omega$.
  Then
\begin{equation}\tag*{$(\mathcal{B}2)_n$}\label{con-b2}
\left\{
\begin{array}{lll}
& \theta_0\notin\cV_{n-1}^+\\
& x_0\notin \inte(C)
\end{array},
\right.
\end{equation}
implies
\begin{equation}\tag*{$(\mathcal{C}2)_n$}
  \theta_{-m}\notin\cW^-_{n-1}\ \follows \ x_{-m}\in \inte(E)  \quad \forall
  m=1,\ldots,\mathcal{R} \ .
\end{equation}
\end{Lemma}
It should be emphasized here that the above two statements are
purely combinatorial in nature, and only rely on the almost
invariance of the contracting and expanding region given by
(\ref{eq:Cinvariance}). If they are combined with quantitative
estimates on the derivates like in (\ref{eq:bounds1})--(\ref{eq:s}),
they can be used to obtain a wealth of further information on
finite-time Lyapunov exponents or the geometry of iterates of
suitable small curves or sets. The basis of such a quantified
analysis are suitable estimates on the proportion of time spent in
the contracting or expanding region.  To that end, given
$\tau,\theta_0,x_0$ and $0\leq m\leq N$, let
\begin{eqnarray}
\mathcal{P}_m^{N} & = & \#\{l\in [m,N-1]: x_l\in C\} \ , \\
\mathcal{Q}_m^{N} & = & \#\{l\in [m,N-1]: x_{-l}\in E\} \ .
\end{eqnarray}
Further, let $\beta_0=1$ and $\beta_n=\prod_{j=0}^{n-1}(
1-K_j^{-1})$. Note that due to the choice of the $K_j$ in Theorem
\ref{t.Main_qualitative} and (\ref{e.t}), we have
\begin{equation}\label{e.beta}
\frac{2}{p}\beta_n-(1-\beta_n)p \ \geq \ \frac{1}{p}
\end{equation}
for all $n\in\N$. Lemmas~\ref{lem-ess} and \ref{lem-ess_back} now
lead to the following quantitative estimates.
\begin{Lemma}[{\cite[Lemma 4.6]{fuhrmann2013NonsmoothSaddleNodesI}}]
  \label{lem-2} Suppose $f_\tau$ satisfies (\ref{eq:Cinvariance}) and
  conditions $(\cX)_{n-1}$ and $(\cY)_{n-1}$ hold.  Let
  $0<L_1<L_2<\ldots<L_J=\mathcal{L}$ denote all those times $L_i\leq
  \mathcal{L}$ for which $\theta_{L_i}\in\mathcal{I}_{n-1}$.  Further,
  assume that \ref{con-b1} holds. Then for each $j=1,\ldots,J$, we
  have
\begin{equation}\label{number_P}
\mathcal{P}_m^{L_j}\geq\beta_n(L_j-m)\ \ \forall m=0,\ldots,L_j-1.
\end{equation}
Further $x_{L_j}\in C,\ \forall j=1,\ldots,J$.

Similarly, let $0<R_1<\ldots<R_J=\mathcal{R}$ denote are all those
times $R_i\leq \mathcal{R}$ for which
$\theta_{-R_i}\in\mathcal{I}_{n-1}+\omega$. Then for each
$j=1,\ldots,J$, we have
\begin{equation}\label{number_Q}
\mathcal{Q}_m^{R_j}\geq\beta_n(R_j-m)\ \ \forall m=0,\ldots,R_j-1.
\end{equation}
Further $x_{-R_j}\in E,\ \forall j=1,\ldots,J$.
\end{Lemma}
These estimates can be used to obtain precise control on the size
and parameter dependence of the critical intervals.

\begin{prop}[{\cite[Proposition 3.11]{jaeger:2009a} and
  \cite[Lemma 4.5]{Jaeger2012SNADiophantine}}]\label{finite_times}
Suppose $\hat f\in \mathcal P_\omega$ satisfies
(\ref{eq:Cinvariance})-(\ref{eq:boundinterval}), $(\cX)_{n-1},\
(\cY)_{n-1}$ hold for some $n\geq 1$ and $\alpha$ is sufficiently
large. Then the two connected components of $\mathcal I_{n}(\tau)$,
denoted as $I_{n}^\iota(\tau)=(a_{n}^\iota(\tau),\
b_{n}^\iota(\tau)),\ \iota=1,2$, are differentiable in $\tau$.
Further, we have
\begin{equation}\label{est_interval}
|I_{n}^\iota(\tau)|\leq \varepsilon_{n},\ \ \iota=1,2,
\end{equation}
\begin{equation}\label{est_relative}
\min\{\partial_\tau a^1_{n}(\tau),\partial_\tau b^1_{n}(\tau)\} -
  \max\{\partial_\tau a^2_{n}(\tau),\partial_\tau b^2_{n}(\tau)\} \ > \
  \ell/S
  \ ,
\end{equation}
\begin{equation}\label{est_partial}
|\partial_\tau I_{n}^\iota(\tau)|\ \leq\  2L/s,\ \ \iota=1,2,
\end{equation}
where $|\partial_\tau I_{n}^\iota(\tau)|=\max\{|\partial_\tau
a_{n}^\iota(\tau)|,\ |\partial_\tau b_{n}^\iota(\tau)|\},\
\iota=1,2$.
\end{prop}
Note that for $n=0$, the respective estimates hold by assumption.
\smallskip

As a first consequence of the above statements, we obtain that the
emptyness of a critical region implies mode-locking.
\begin{prop} \label{p.modelocking_criterion}
  The constants $\alpha_*'$ and $\eps_*'$ in
  Theorem~\ref{t.Main_qualitative} can be chosen such that if
  $\alpha>\alpha_*$ and $\eps_0<\eps_*$, then the following holds.

  Let $K_0\ld K_n$ be chosen as in Theorem \ref{t.Main_qualitative}.
  Further, suppose that for some $\tau\in\kreis$ the numbers
  $M_0\ld M_n$ can be chosen such that (\ref{e.M_j}) holds for $j=0\ld
  n-1$ and conditions $(\cX)_{n-1}$ and $(\cY)_n$ are satisfied, but
  $\cC_{n}=\emptyset$.  Then $f_\tau$ has an attracting continuous
  invariant graph.  In particular, $f_\tau$ is mode-locked.
\end{prop}
\proof For convenience, we omit the parameter $\tau$ throughout the
proof. First, by Proposition~\ref{finite_times}, we have
$|I_j^\iota|\leq \varepsilon_j,\ \iota=1,2,\ j=0\ld n$. Then by
(\ref{e.M_j}), (\ref{e.eps_j}) and (\ref{e.Z_nV_n}), we know that
$\mathcal W_n^+,\ \mathcal V_n^-$ are unions of small intervals which
satisfy the following estimates
\begin{equation}\label{esti_v_n}
\mathrm{Leb}(\mathcal W^+_n)\leq
\sum_{j=0}^n(M_j+1)\varepsilon_j<2M_0\varepsilon_0+
\frac{8}{s}\cdot\sum_{j=1}^{n}\alpha^{-\frac{M_{j-1}}{4p}}<\frac{1}{2(p^2+2)},
\end{equation}
\begin{equation}
\mathrm{Leb}(\cV_n^-)\leq \sum_{j=0}^n M_j\varepsilon_j\leq
M_0\varepsilon_0+
\frac{4}{s}\cdot\sum_{j=1}^{n}\alpha^{-\frac{M_{j-1}}{4p}}<\frac{1}{4(p^2+2)},
\end{equation}
for $\alpha$ large and $\eps_0$ small. Thus, there must be some
interval $\mathcal{J}'\subseteq\T^1\setminus(\cV_{n}^-\cup\cW^+_{n})$.
We let $\mathcal{J}'=(a',b')$ and $\lambda=|\mathcal{J}'|>0$. Let
$\mathcal{J}=[a'+\lambda/3,\ b'-\lambda/3]$. Since $\omega$ is
irrational, there must be some $K\in\N$ such that
$\inte(\mathcal{J}+K\omega)\cap \inte(\mathcal{J})\neq\emptyset$ and
$b'-\lambda/3\in \inte(\mathcal{J}+K\omega)$. In particular, we have
$\mathcal{J}+K\omega\subseteq\mathcal{J}'$. Since $(\cY)_n$ holds,
$\mathcal I_{n+1}=\emptyset$ and $\mathcal J\cap
(\cW^+\cup\cV^-)=\emptyset$, Lemma \ref{lem-ess} implies
\begin{equation}
  f^K(\mathcal{J}\times C)\ \subseteq \ (\mathcal{J}+K\omega)\times C \ .
\end{equation}
Hence, we obtain $ f^K(\mathcal{J}\times C)\cap (\mathcal{J}\times
C)\neq\emptyset$, and thus
\begin{equation}\label{equ-21}
f^{(j+1)K}(\mathcal{J}\times C)  \cap  f^{jK}(\mathcal{J}\times
C)\neq\emptyset,\ \ j=1,2,\ldots
\end{equation}
Moreover, there exists $N>1$, such that
$\inte(\mathcal{J}+NK\omega)\cap \inte(\cal{J})\neq\emptyset$,
$a'+\lambda/3\in\inte(\mathcal{J})+NK\omega$ and $a'+\lambda/3\notin
\mathcal J+(N+1)K\omega$. Then we have
$\cup_{j=0}^{N}(\mathcal{J}+jK\omega)=\T^1$. By the same reasoning as
above, we further have that
$f^{NK}(\mathcal{J}\times C)\subseteq (\mathcal{J}+NK\omega)\times
C$, and
\begin{equation}\label{equ-c-1-2}
f^{NK}(\mathcal{J}\times C)\cap (\mathcal{J}\times C)\neq\emptyset.
\end{equation}
Consequently, the set
\begin{equation}
\mathcal{A}:=\bigcup_{j=0}^Nf^{jK}(\mathcal{J}\times C)
\end{equation}
is connected and wraps around the torus in the horizontal direction.
In fact, if we assume $N$ to  be minimal with the above property,
$\mathcal A$ horizontally wraps around the torus exactly once. We
now claim that $f^{(N+1)K}(\mathcal J\times C)\subseteq (\mathcal
J\times C)\cup f^K(\mathcal J\times C)$, which immediately implies
\begin{equation}
f^K(\mathcal A) \ \subseteq \ \mathcal A \ .
\end{equation}
The reason is the following. Suppose $(\theta,x)\in
f^{(N+1)K}(\mathcal J\times C)$. Then since $d(K\omega, 0)<|\mathcal
J|$ and due to the choice of $N$ above, there are two possibilities.
On the one hand, we may have $\theta\in \mathcal J$. In this case,
the fact that $\mathcal J\cap (\mathcal W^+_n\cup \mathcal
V_n^-)=\emptyset$ implies, via Lemma \ref{lem-ess}, that
$(\theta,x)\in \mathcal J\times C$. On the other hand, we may have
$\theta-K\omega\in \mathcal J$. Then the same argument yields
$f^{-K}(\theta,x)\in \mathcal J\times C$, and thus $(\theta,x)\in
f^K(\mathcal J\times C)$. In both cases, we have $(\theta,x)\in
(\mathcal J\times C)\cup f^K(\mathcal J\times C)$.

Since $\mathcal W^+_n$ is a finite union of small intervals and
$\omega$ is irrational, then by Weyl's criterion,
$\{\theta_0+m\omega\}_{m\in
  \N}$ is equidistributed in $\kreis$ for all $\theta_0\in\kreis$,
which means that
\begin{equation}\label{LE}
  \lim_{m\rightarrow
    \infty}\frac{1}{m}\sum_{j=0}^{m-1}\textbf{1}_{\cW_n^+}(\theta_0+m\omega)
   =\mathrm{Leb}(\mathcal W^+_n).
\end{equation}
Let $(\theta_0,x_0)\in \mathcal J \times C$.  Using Lemma
\ref{lem-ess} in combination with (\ref{LE}), (\ref{eq:bounds1}) and
(\ref{eq:bounds3}), we obtain
$$\varlimsup \ntel\log\partial_x f^n_{\theta_0}(x_0) \ \leq \ 
  (-2/p+(2/p+p)\mathrm{Leb}(\mathcal W^+_n))\log \alpha
\ \stackrel{\eqref{esti_v_n}}{\leq} \  -\log \alpha /p \  .$$ By the
definition of $\mathcal A$, it is now easy to show that all points in
$\mathcal A$ have negative vertical Lyapunov exponents. By
\cite[Corollary 1.15]{sturman/stark:2000}, this implies that the
compact invariant set $\bigcap_{n\in \N}f^{nK}(\mathcal A)$ is the
graph of a continuous curve with negative vertical Lyapunov exponent.
Since this implies mode-locking~\cite{bjerkloev/jaeger:2009}, the
proof is complete.  \qed\medskip

One task which will frequently come up in the proof of the main
theorem is to control the geometry of small arcs whose iterates
remain in the contracting region (resp. expanding region) most of
the time. The following statements cover all these situations.
\begin{Lemma}[Forwards Iteration]\label{estimate.forward} Suppose $f_\tau$ satisfies assumptions
  (\ref{eq:Cinvariance}), (\ref{eq:bounds1}), (\ref{eq:bounds3}),
  (\ref{eq:bounddth}), (\ref{eq:s}) and the slow recurrence conditions
  $(\cX)_{n-1}$ and $(\cY)_{n-1}$ hold. Let $I\subset\T^1$ be an interval and
  $N\geq 1$. Then, if $\alpha$ is sufficiently large and $\eps_0$ is
  sufficiently small, the following are true. \smallskip

  If $\phi^1: I\rightarrow\T^1\smin\inte(E)$ is a $\cC^1$-curve and
\begin{equation}\label{con-d1}
\tag*{$(\mathcal{D}1)_n$}\left\{
\begin{array}{llll}&I\cap\cV_{n-1}^-=\emptyset,\\
&I+N\omega\subset \mathcal{I}_{n-1},\\
&(I+l\omega)\cap \mathcal{I}_n=\emptyset, \ \forall\ l=0,1,\ldots,N-1,\\
\end{array}\right.
\end{equation}
then we have
\begin{equation}\label{con-e1}
\left|\partial_\theta f^N_{\tau,\theta}
\left(\phi^1(\theta)\right)\right|\ \leq\
\sum_{l=0}^{N-1}\alpha^{-l/p}S+\alpha^{-N/p}\left|\partial_\theta\phi^1(\theta)\right|\
.\
\end{equation}
Further, we also have
\begin{itemize}
\item[$(i)$] if $I+N\omega\subset I_{n-1}^1$, then
\begin{equation}\label{equ-5}
\begin{split}
-S-\frac{S}{\alpha^{1/p}-1}-\alpha^{-\frac{N+1}{p}}&
\left|\partial_\theta\phi^1(\theta)\right|\
 \ \leq\
\partial_\theta f_{\tau,\theta}^{N+1}(\phi^1(\theta))\  \\ & \ \leq \
-s+\frac{S}{\alpha^{1/p}-1}+\alpha^{-\frac{N+1}{p}}\left|\partial_\theta\phi^1(\theta)\right|
\ ;
\end{split}
\end{equation}
\item[$(ii)$] if $I+N\omega\subset I_{n-1}^2$, then
\begin{equation}\label{equ-6}
  \begin{split} s-\frac{S}{\alpha^{1/p}-1}-\alpha^{-\frac{N+1}{p}}
  &\left|\partial_\theta\phi^1(\theta)\right|
    \ \leq \
\partial_\theta f_{\tau,\theta}^{N+1}(\phi(\theta))\   \\ & \ \leq \
S+\frac{S}{\alpha^{1/p}-1}+\alpha^{-\frac{N+1}{p}}\left|\partial_\theta\phi^1(\theta)\right|
\ .
\end{split}
\end{equation}
\end{itemize}
Moreover, if $\phi^1,\phi^2: I\rightarrow\T^1\smin\inte(E)$ are
$\cC^1$-curves and \ref{con-d1} holds, then
\begin{equation}\label{con-e2}
\left|f_{\tau,\theta}^j(\phi^1(\theta))-f_{\tau,\theta}^j(\phi^2(\theta))\right|\
\leq\  \alpha^{-j/p}\left|\phi^1(\theta)-\phi^2(\theta)\right|\
\quad \textrm{for}\ \ j=N,N+1.
\end{equation}

\end{Lemma}
\begin{proof} Again, we omit the parameter $\tau$ during the proof. Moreover, we
  assume that the parameter $\alpha$ is sufficiently large, and all estimates
  below should be understood under this premise. For any $m\geq 1$, $\theta\in
  I$ and $\iota=1,2$, we let
  $\phi_m^\iota(\theta)=f_\theta^m(\phi^\iota(\theta))$.  Set
  $\theta_0:=\theta\in I$ and $x_0:=\phi^\iota(\theta_0)\notin \inte(E)$.  Then we have
\begin{equation}\label{equ-1}
\begin{split}
\partial_\theta\phi_m^\iota(\theta)& \ = \
\left(\partial_\theta
f_{\theta_{m-1}}\right)(x_{m-1})+\left(\partial_x
f_{\theta_{m-1}}\right)(x_{m-1})\cdot \partial_\theta
f_{\theta_0}^{m-1}\left(\phi^\iota(\theta_0)\right)=\cdots\\
 &\ = \ \partial_\theta
f_{\theta_{m-1}}(x_{m-1})+\sum_{l=0}^{m-2}\big(\partial_x
f_{{\theta_{l+1}}}^{m-l-1}\big)(x_{l+1})\cdot(\partial_\theta
f_{\theta_l})(x_l) \\
& \ \ +\big(\partial_x f_{\theta_0}^m\big)(x_0)\cdot\partial_\theta
\phi^\iota(\theta),
\end{split}
\end{equation}
where
\begin{equation}
\big(\partial_xf_{\theta_{l+1}}^{m-l-1}\big)(x_{l+1}) \ = \
\prod_{j=l+1}^{m-1}\big(\partial_x f_{\theta_j}\big)(x_j),\ \
l=-1,0,\ldots,m-2.
\end{equation}
Taking $m=N$ we can apply Lemma \ref{lem-2}, whose conditions hold
due to \ref{con-d1}. We thus obtain
$$\mathcal{P}_{l+1}^N \ \geq \  \beta_n(N-l-1),$$
which implies that
\begin{equation}\begin{split}\nonumber
\left|\left(\partial_xf_{\theta_{l+1}}^{N-l-1}\right)(x_{l+1})\right|
& \ \leq \ \alpha^{-\frac{2}{p}\mathcal{P}_{l+1}^N}
\alpha^{p(N-l-1-\mathcal{P}_{l+1}^N)} \\ & \ \leq \
\left(\alpha^{-\frac{2}{p}\beta_n+(1-\beta_n)p}\right)^{N-l-1}\
\stackrel{\eqref{e.beta}}{\leq} \ \alpha^{-\frac{N-l-1}{p}}\ .
\end{split}
\end{equation} As $|\partial_\theta f_{\theta_l}|\ \leq
\ S,\ \forall l$ by \eqref{eq:bounddth}, this yields the estimate
\eqref{con-e1}. Further, since
$$\left|\phi_N^1(\theta)-\phi_N^2(\theta)\right|\ = \ \left|\partial_x
  f_{\theta_0}^N(\xi_0)\right|\cdot
\left|\phi^1(\theta)-\phi^2(\theta)\right|\ \leq \
\alpha^{-N/p}\left|\phi^1(\theta)-\phi^2(\theta)\right|$$ for some
$\xi_0\notin \inte(E)$ between $\phi^1(\theta)$ and $\phi^2(\theta)$,
we also obtain \eqref{con-e2} for $j=N$ in the same way. In order to
show (\ref{con-e2}) for $j=N+1$, note that
$\left[\phi_N^1(\theta),\phi_N^2(\theta)\right]\subseteq C$ by Lemma
\ref{lem-2}. There exists
$\eta\in\left[\phi^1_N(\theta),\phi^2_N(\theta)\right]$ such that
\begin{eqnarray*}\lefteqn{
|f_{\theta_0+N\omega}(\phi_N^1(\theta))-f_{\theta_0+N\omega}(\phi_N^2(\theta))|
  \ = \ |\partial_xf_{\theta_0+N\omega}(\eta)||\phi_N^1(\theta)-\phi_N^1(\theta)|}\\
&\stackrel{\eqref{eq:bounds3}}{\leq}&\alpha^{-2/p}|\phi_N^1(\theta)-\phi_N^2(\theta)|
\ \leq \ \alpha^{-\frac{N+1}{p}}|\phi^1(\theta)-\phi^2(\theta)| \ .
\end{eqnarray*}
Finally, in order to show (i), suppose $I+N\omega\subset I_{n-1}^1$.
Then we have
$$\partial_\theta f_\theta^{N+1}(\phi^\iota(\theta))=(\partial_\theta f_{\theta_N})(x_N)+(\partial_x f_{\theta_N})(x_N)
\partial_\theta \phi_N^{\iota}(\theta).$$ Since $\theta_N\in I_{n-1}^1\subset
I_0^1$ and $x_N=\phi_N^\iota(\theta)\in C$, we obtain~\eqref{equ-5}
from \eqref{con-e1}, (\ref{eq:bounddth}) and (\ref{eq:s}), provided that $\alpha$ is
large enough. The proof of (ii) is analogous.
\end{proof}

A similar statement holds for the backwards iteration.

\begin{Lemma}[Backwards Iteration]\label{estimate.back} Suppose $f$ satisfies the
  conditions (\ref{eq:Cinvariance})-(\ref{eq:bounds2}), (\ref{eq:bounddth}),
  (\ref{eq:s}) and the slow recurrence conditions $(\mathcal{X})_{n-1}$ and
  $(\mathcal{Y})_{n-1}$ hold. Let $I\subset\T^1$ be an interval and $N\geq 1$.
  Then, if $\alpha$ is sufficiently large and $\eps_0$ is sufficiently small,
  the following are true. \smallskip

  If $\phi^1,\phi^2:I\to\kreis\smin\inte(C)$ are $\mathcal C^1$-curves and
\begin{equation}\label{con-d2}
\tag*{$(\mathcal{D}2)_n$}\left\{
\begin{array}{llll}&I\cap\cV_{n-1}^+=\emptyset,\\
&I-N\omega\subset \mathcal{I}_{n-1}+\omega,\\
&(I-l\omega)\cap (\mathcal{I}_n+\omega)=\emptyset, \ \forall\ l=0,1,\ldots,N-1,\\
\end{array}\right.
\end{equation}
then we have
\begin{equation}\label{con-e3}
\left|\partial_\theta
f_\theta^{-N}\left(\phi^\iota(\theta)\right)\right| \ \leq\
\sum_{l=1}^{N}\alpha^{-l/p}S+\alpha^{-\frac{N+1}{p}}\left|\partial_\theta\phi^\iota(\theta)\right|\
,\quad \iota=1,2,
\end{equation}
and
\begin{equation}\label{con-e4}
\left|f_\theta^{-N}\left(\phi^1(\theta)\right)-f_\theta^{-N}\left(\phi^2(\theta)\right)\right|\
\leq \ \alpha^{-N/p}\left|\phi^1(\theta)-\phi^2(\theta)\right| \ .
\end{equation}

\end{Lemma}

\begin{proof}
  As before, we omit $\tau$. For $\iota=1,2$, let
  $\phi_{-N}^\iota(\theta)=f_\theta^{-N}\left(\phi^\iota(\theta)\right),\
  \theta\in I$. Further, let $\theta_0:=\theta\in I$ and
  $x_0:=\phi^\iota(\theta_0)\notin \inte(C)$. We proceed in a similar way as in
  Lemma~\ref{estimate.forward}, but this time consider the map $f^{-1}$ instead
  of $f$. Thus, we write $\theta_l=\theta_0-l\omega$ and
  $x_l=f_{\theta_0}^{-l}(x_0).$ First, note that
\begin{equation}
  \partial_xf_\theta^{-1}(x)\ = \ \frac{1}{(\partial_xf_{\theta-\omega})(f_\theta^{-1}(x))} \ \in
  \ (\alpha^{-p},\alpha^{-2/p}) \quad \textrm{ if } f^{-1}_\theta(x)\in E
\end{equation}
by \eqref{eq:bounds1}, \eqref{eq:bounds2} and
\begin{equation}
\partial_\theta f_\theta^{-1}(x)=-\frac{(\partial_\theta
f_{\theta-\omega})(f_\theta^{-1}(x))}{(\partial_xf_{\theta-\omega})(f_\theta^{-1}(x))}
\ .
\end{equation}
Similarly to $(\ref{equ-1})$, we have
\begin{eqnarray*}
\partial_\theta\phi_{-N}^\iota(\theta)\ =\ \big(\partial_x
f_{\theta_0}^{-N}\big)(x_0)\cdot\partial_\theta
\phi^\iota(\theta_0)+\sum_{l=0}^{N-1}\big(\partial_x
f_{\theta_{l+1}}^{-(N-l-1)}\big)(x_{l+1})\cdot(\partial_\theta
f_{\theta_l}^{-1})(x_l).
\end{eqnarray*}
 Since condition \ref{con-d2} holds, Lemma~\ref{lem-2} yields
$$\mathcal{Q}_{l+1}^N\geq \beta_n(N-l-1),\ \ l=-1\ld N-1.$$
Then
\begin{eqnarray*}
\lefteqn{\left|\left(\partial_x
f_{\theta_{l+1}}^{-(N-l-1)}\right)(x_{l+1})\left(\partial_\theta
f_{\theta_l}^{-1}\right)(x_l)\right|\ = \
\left|\prod_{j=l+1}^{N-1}\left(\partial_x
f_{\theta_j}^{-1}\right)(x_j)\frac{\left(\partial_\theta
f_{\theta_{l+1}}\right)(x_{l+1})}{\left(\partial_x
f_{\theta_{l+1}}\right)(x_{l+1})}\right|} \\
&=&\left|\prod_{j=l}^{N-1}\left(\partial_x
f_{\theta_j}^{-1}\right)(x_j)\cdot\left(\partial_\theta
f_{\theta_{l+1}}\right)(x_{l+1})\right|\ \leq \
\alpha^{-\frac{2}{p}}\alpha^{-\frac{2}{p}\mathcal Q_{l+1}^N}
\alpha^{p(N-l-1-\mathcal Q_{l+1}^N)}S \\
&\leq&\alpha^{-\frac{2}{p}}\left(\alpha^{-\frac{2}{p}\beta_n+(1-\beta_n)p}\right)^{N-l-1}S
\ \leq\  \alpha^{-\frac{N-l}{p}}S
\end{eqnarray*}
for $l=-1,0,\ldots,N-1$. This implies \eqref{con-e3}. The estimate
\eqref{con-e4} is obtained in a similar way as \eqref{con-e2}.
\end{proof}

\begin{rem}  \label{esti-comm}
  By equality (\ref{equ-1}), for any $\mathcal C^1$-curves $\phi^1,\phi^2$
  defined on an interval $I\subset \T^1$, $m\geq 1$, we obtain
  that \begin{equation}\label{equ-8} |\partial_\theta
    f_\theta^m(\phi^\iota(\theta))|\ \leq\
    \sum_{l=0}^{m-1}\alpha^{pl}S+\alpha^{pm}|\partial_\theta\phi^\iota(\theta)|\
    ,\quad \iota=1,2,
\end{equation}
and
\begin{equation}\label{equ-7}
\alpha^{-pm}\left|\phi^{1}(\theta)-\phi^{2}(\theta)\right|\ \leq \
\left|f_\theta^m\left(\phi^{1}(\theta)\right)-f_\theta^m\left(\phi^{2}(\theta)\right)\right|\
\leq \ \alpha^{pm} \left|\phi^{1}(\theta)-\phi^{2}(\theta)\right| \
,
\end{equation}
provided $f$ satisfies conditions (\ref{eq:bounds1}) and
(\ref{eq:bounddth}).
\end{rem}

\section{Geometric estimates and the proof of
  Theorem~\ref{t.modelocking_quantitative}} \label{GeometricEstimates}

In this section, we collect the key technical lemmas about the
geometry of the intersections shown in Figure~\ref{Fig.SNA} and show
how this information can be combined to prove
Theorem~\ref{t.modelocking_quantitative}. The proofs of the lemmas
will then be given in Section~\ref{MainProof}.

Recall that our main aim is to render the critical set $\cC_n$ empty
by shifting the parameter $\tau\in\Lambda^{\hat f}$. As mentioned in
Section~\ref{ProofOutline}, the first step is to create a fast return
of $\cI_n$ to itself. Thereby, it will be important to ensure that the
following condition, which is an itermediate between $(\cY)_n$ and
$(\cY')_n$, still holds.
\begin{equation}
  \label{eq:Y''} \tag*{$(\cY'')_{n}$} d\left((\cI_j-(M_j-1)\omega)
    \cup(\cI_j+(M_j+1)\omega),\cY_{j-1}\right) \ > \ \eps_{j-1} \quad \forall j=1\ld n .
  \end{equation}
  \begin{lem} \label{l.perturbation_closereturn} Let $\hat f$ satisfy the
    assertions of Theorem~\ref{t.Main_qualitative}, assume that
    $\tau_0\in\Lambda^{\hat f}$ and fix the corresponding sequences $M_n$ and
    $\varepsilon_n$. Then for all $\zeta>0$ there exist integers $n\in\N$, $k\in
    [2K_{n-1}M_{n-1}+1,\ M_{n-1}^{4q(\nu+1)}]$ and an interval
    $\Gamma=[\tau^-,\tau^+]\ssq B_\zeta(\tau_0)$ such that for all
    $\tau\in\Gamma$ the following hold.  \romanlist
  \item Conditions $(\cX)_{n-1}$ and $(\cY'')_n$ are satisfied.
\item The intervals $I^1_n+k\omega$ and $I^2_n$ have distance no
more
 than $4\eps_n$.
\item At $\tau=\tau^-$, the interval $I_n^1+k\omega$ is to the left
  of $I_n^2$, whereas at $\tau=\tau^+$ it is to the right (in a local
  sense).  \listend
\end{lem}
\begin{rem}
  Note that due to the assumptions on the sequences $K_n$ and $M_n$ in
  Theorem~\ref{t.Main_qualitative}, we have $M_{n-1}\ll k \ll M_n$ if
  $\alpha$ and $n$ are large.
\end{rem}

For any $\tau\in\Gamma$, we define $J=J(\tau)$ by
\begin{equation}
  \label{e.J_definition}
  J \ = \ \closure\left(B_{4\eps_n}(I_n^1+k\omega)\cup B_{4\eps_n}(I_n^2)\right) \ .
\end{equation}
Note that due to statement (ii) in
Lemma~\ref{l.perturbation_closereturn}, $J$ is always an interval.
We further let
\begin{equation}
  \label{e.A'_B'}
  \cA' \ = \ (J-(M_n-1)\omega)\times C \eqand \cB' \ =
  \ (J+(M_n-k+1)\omega)\times E \ .
\end{equation}
The overall strategy from now on is illustrated and outlined
Figure~\ref{f.stragegy}. 
 \begin{figure}[ht!] 
 \epsfig{file=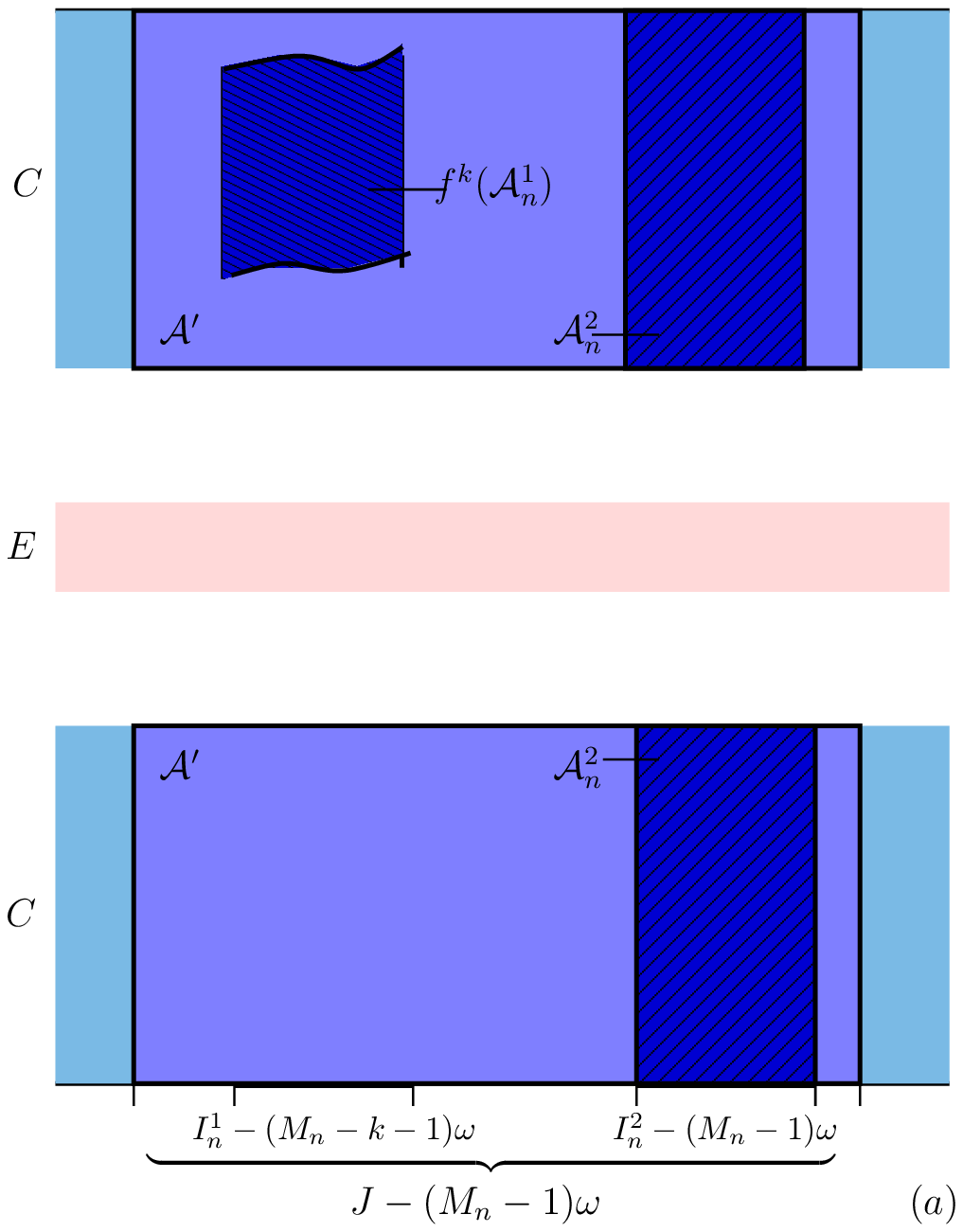,
 clip=,width=0.47\linewidth} \hspace{3eM}
 \epsfig{file=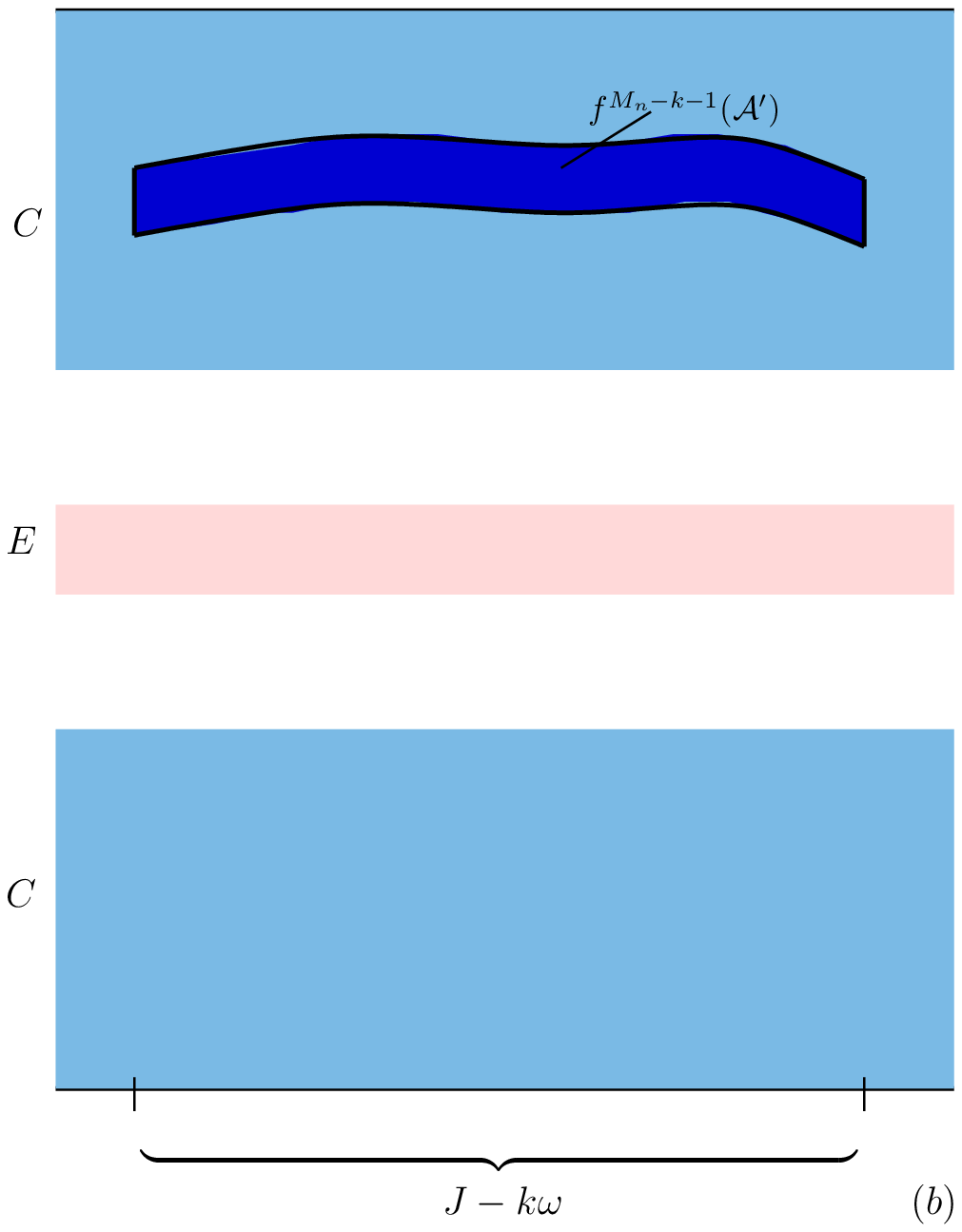, clip=,width=0.47\linewidth}
\vspace{2ex}

\epsfig{file=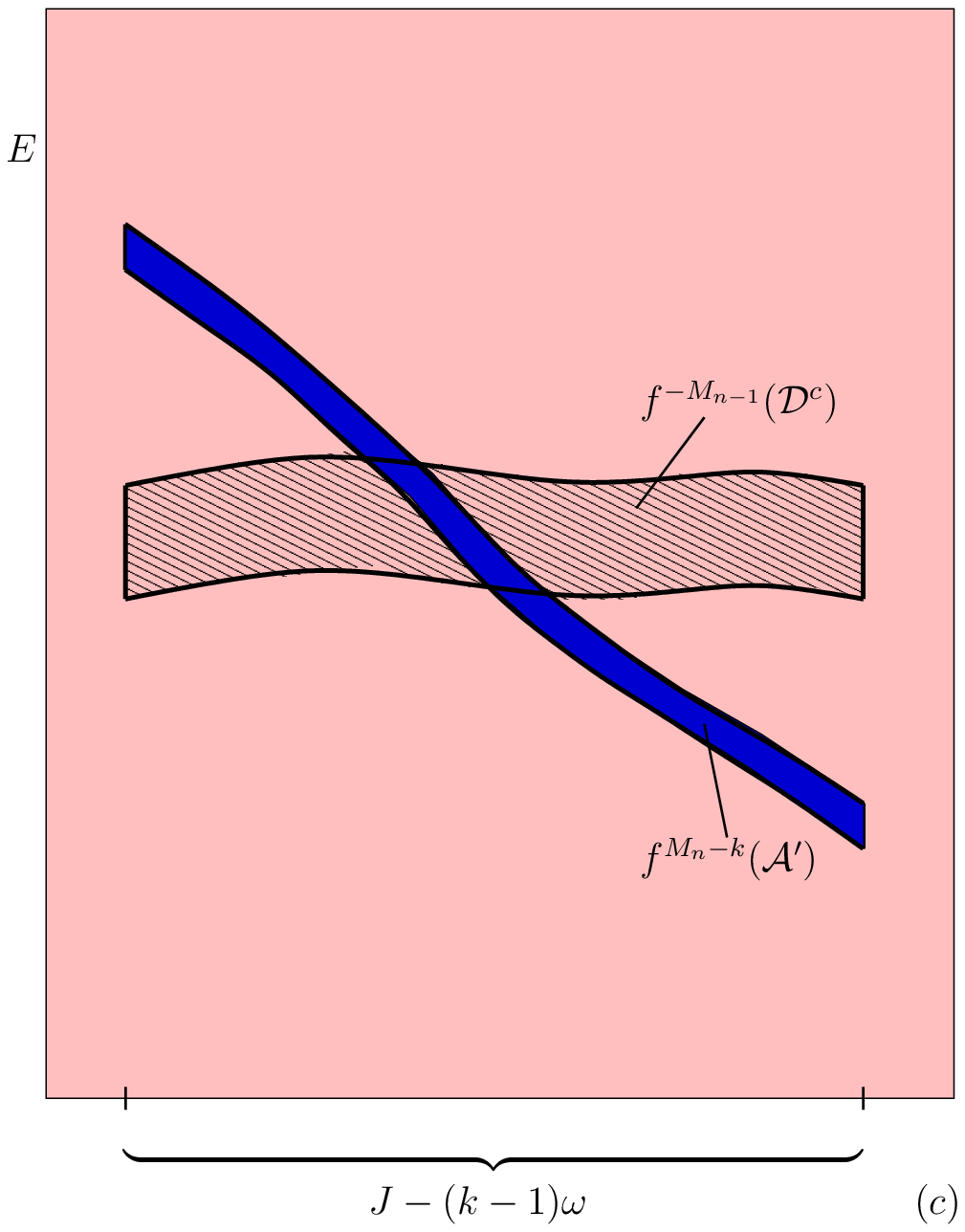,
clip=,width=0.47\linewidth}\hspace{3eM}
 \epsfig{file=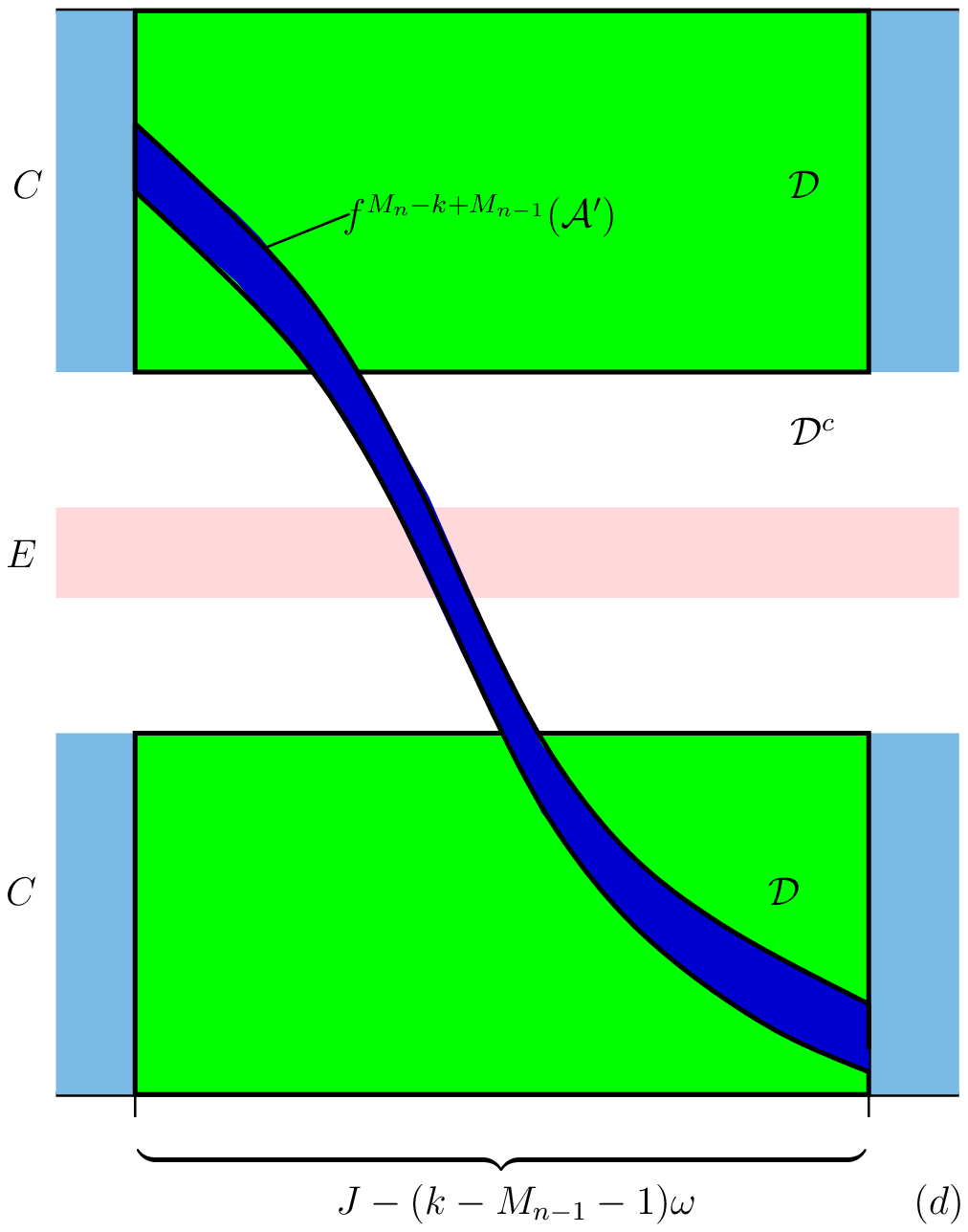, clip=,width=0.47\linewidth}
 \vspace{2ex}

 \caption{\small Strategy for the proof of
   Theorem~\ref{t.modelocking_quantitative}: The different steps in
   the forward iteration of $\cA'$, explaining the creation of the two
   hooks in Figure~\ref{Fig.SNA}.  (a) It suffices to consider the
   sets $\cA'$ and $\cB'$ defined in (\ref{e.A'_B'}), since $\cA'$
   contains $\cA^2_n\cup f^k(\cA^1_n)$, and similarly $\cB'$ contains
   $\cB^1_n\cup f^{-k}(\cB^2_n)$ (Lemma~\ref{l.inclusions}). (b) After
   $M_n-k-1$ iterates, the image of $\cA'$ is a thin horizontal strip
   in the contracting region $\kreis\times C$. (c) In the next step,
   it is mapped into the expanding region $\kreis\times E$ with
   negative slope. Therefore it intersects the preimage of the
   complement of $\cD$ under $f^{M_{n-1}}$ in a transveral way. (d)
   After $M_{n-1}$ further steps, the image of $\cA$ is mostly
   contained in $\cD$, but transverses the expanding region in a small
   interval. Continued in Figure~\ref{f.stragegy2} \ldots\
   .\label{f.stragegy}}
\end{figure}

The following lemma ensures that it is sufficient to consider $\cA'$
and $\cB'$ (instead of the four sets $\cA^1_n,\cA^2_n,\cB^1_n$ and
$\cB^2_n$).
\begin{lem} \label{l.inclusions}
  For all $\tau\in\Gamma$, the following inclusions hold.
  \begin{eqnarray}
    f^{M_n}(\cA^2_n) \cap f^{-M_n}(\cB^2_n) & \ssq & f^{M_n}(\cA')\cap f^{-(M_n-k)}(\cB')\\
    f^k\left(f^{M_n}(\cA^1_n) \cap f^{-M_n}(\cB^1_n)\right) &
    \ssq & f^{M_n}(\cA')\cap f^{-(M_n-k)}(\cB')
  \end{eqnarray}
\end{lem}
Hence, in order to apply Proposition~\ref{p.modelocking_criterion}
it will be sufficient to show that $f^{M_n}(\cA')\cap
f^{-(M_n-k)}(\cB')=\emptyset$, since in this case both components of
$\cC_{n}$ are empty.
 
 \begin{figure}[b!] 
 \epsfig{file=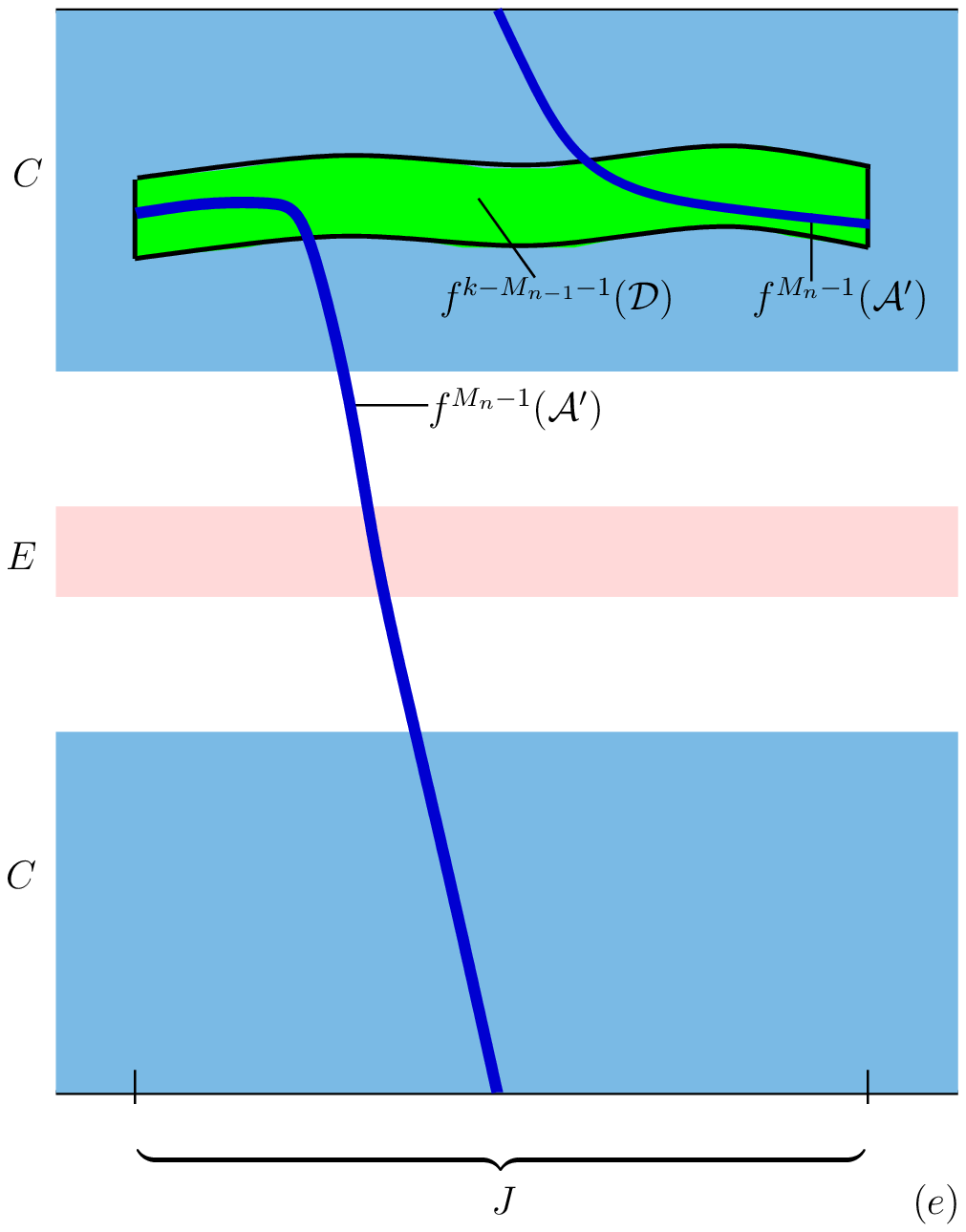,
 clip=,width=0.47\linewidth}\hspace{3eM}
 \epsfig{file=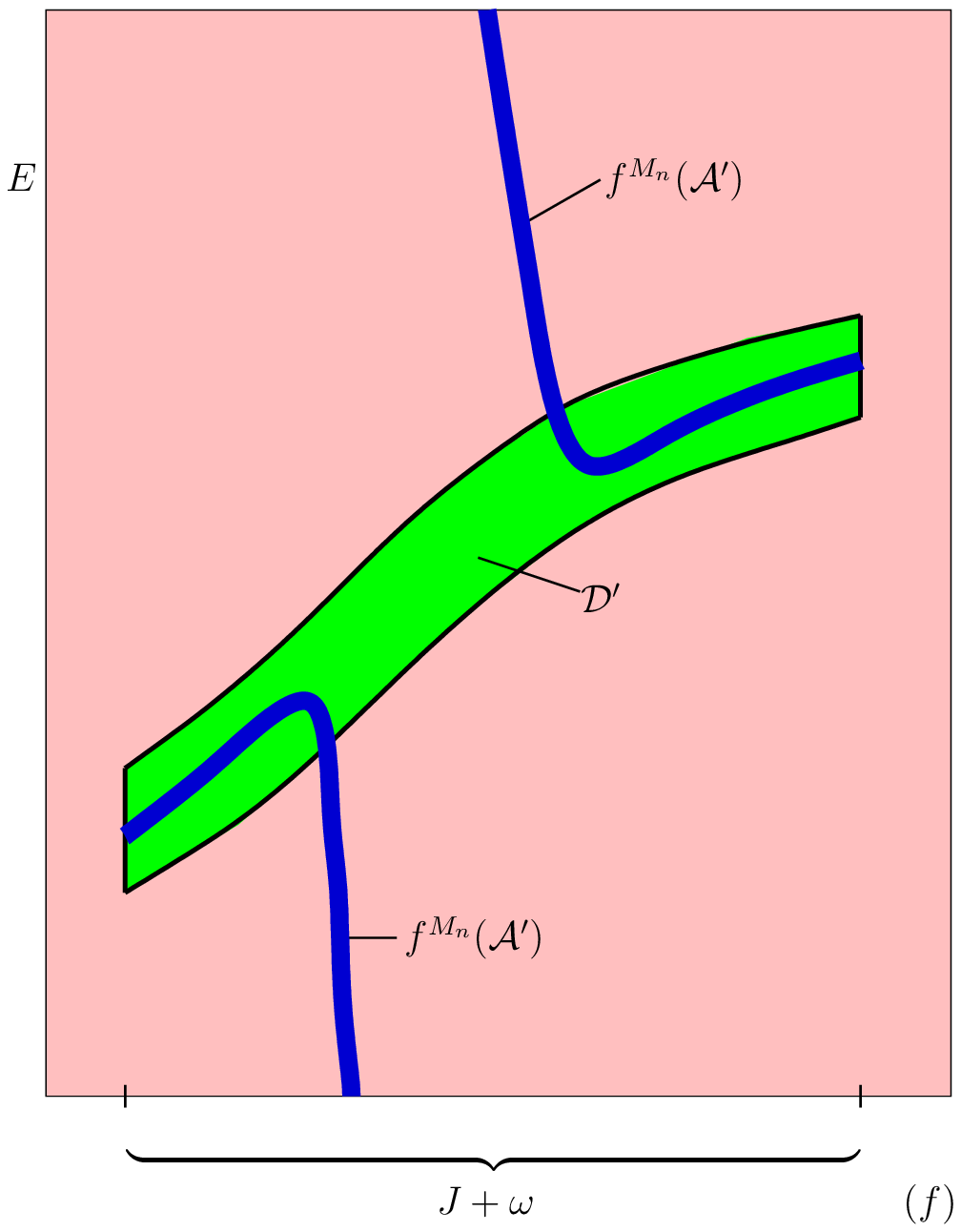, clip=,width=0.47\linewidth}
 \vspace{2ex}
 \caption{\small Strategy for the proof of
   Theorem~\ref{t.modelocking_quantitative}: (e) After $k-M_{n-1}-1$ more iterates the set
   $\cD$ gets mapped to a thin horizontal strip in the contracting region. (f)
   In the next step, it gets mapped into the expanding region with positive
   slope (Lemma~\ref{l.D_shape}). Due to the relative position of the two sets,
   this forces the image of $\cA'$ to develop the two hooks already mentioned in
   Figure~\ref{Fig.SNA}(b) and (c) (Lemma~\ref{l.P_0}). \label{f.stragegy2}}
\end{figure}

Next, it will be important to control the geometry of the sets $f^{M_n}(\cA')$
and $f^{-(M_n-k)}(\cB')$. To that end, we introduce the
following notation.  If $I\ssq\kreis$ is an interval and $A\ssq\kreis$, we
denote by $\sup^I A$ and $\inf^I A$ the supremum, respectively infimum, of $A$
with respect to the natural ordering on $I$, induced by the counter-clockwise
orientation on \kreis. Note that thus $\inf^I I$ and $\sup^I I$ are the left and
right endpoints of $I$.  Given $A\ssq\T^2$, $\theta\in \pi_1(A)$, we let
$A_\theta=\{x\in\kreis\mid (\theta,x)\in A\}$. If $A_\theta$ is an interval for
all $\theta\in \pi_1(A)$, we define the {\em boundary graphs} of $A$ as
\begin{eqnarray*}
  \varphi^+_A : \pi_1(A) \to \kreis & , & \varphi^+_A(\theta)\ = \ \textstyle\sup^{A_\theta}A_\theta \\
  \varphi^-_A : \pi_1(A) \to \kreis & , & \varphi^-_A(\theta) \ = \ \textstyle\inf^{A_{\theta}} A_\theta \ .
\end{eqnarray*}
With these notions, we have
\begin{lem} \label{l.preimage_shape} For all $\tau\in\Gamma$, the set
  $\cB''=\closure(f_\tau^{-(M_n-k)}(\cB'))$ is included in
  $(J+\omega)\times E$ and satisfies the following.  \romanlist
\item $|\cB''_\theta|\  \leq \ |E|\cdot\alpha^{-\frac{M_n-k}{p}}$ for all $\theta\in J+\omega$;
\item $|\partial_\theta\varphi^\pm_{\cB''}(\theta)|\  \leq\
  \frac{S}{\alpha^{1/p}-1}$ for all $\theta\in J+\omega$.  \listend
\end{lem}

\begin{lem} \label{l.A'image_shape} For all $\tau\in\Gamma$, the set
  $\cA''=\closure(f_\tau^{M_n}(\cA'))$ satisfies the following.
  \romanlist
\item $|\cA''_\theta|\ \leq \
  |C|\cdot\alpha^{-\frac{M_n}{p}+(p+\frac{1}{p})k}$ for all $\theta\in
  J+\omega$;
\item $|\partial_\theta\varphi^\pm_{\cA''}(\theta)| \ \leq
\  \alpha^{pk}\cdot \frac{2S}{1-\alpha^{-p}}$ for all $\theta\in
J+\omega$. \listend
\end{lem}

Further, the crucial step in the argument will be to control the position of
$f^{M_n}(\cA')$ with respect to an intermediate set $\cD'$ that is defined as
follows. Let
\begin{equation}
  \label{e.D}
  \cD \ = \ (J-(k-M_{n-1}-1)\omega)\times C \eqand \cD'=f^{k-M_{n-1}}(\cD) \ .
\end{equation}
Concerning the geometry of $\cD'$ itself, we have
\begin{lem}\label{l.D_shape}
  For all $\tau\in\Gamma$, the set $\cD'$ satisfies the following
  assertions.  \romanlist
\item $\cD'\ssq (J+\omega)\times E$;
\item $|C|\cdot\alpha^{-p(k-M_{n-1})}\ \leq \ |\cD'_\theta|\ \leq \
  |C|\cdot\alpha^{-\frac{k-M_{n-1}}{p}}$ for all $\theta\in J+\omega$;
\item $s-\frac{S}{\alpha^{1/p}-1}\ \leq\
  \partial_\theta\varphi^\pm_{\cD'}(\theta)\ \leq\
  S+\frac{S}{\alpha^{1/p}-1}$ for all $\theta\in J+\omega$. \listend
\end{lem}
Now, the following statements yield the required information about the relations
between $\cD'$ and $f^{M_n}(\cA')$.
\begin{lem} \label{l.P_0} For all $\tau\in\Gamma$ there exists an open
  interval $P_0\ssq J+\omega$ of length between $\frac{1-|C|}{4S}\cdot
  \alpha^{-pM_{n-1}}$ and $\frac{4(1-|C|)}{s}\cdot\alpha^{-M_{n-1}/p}$
  such that $\pi_1(\cA''\cap \cD') = (J+\omega)\smin P_0$.  Moreover,
  $\cA''$ leaves and enters $\cD'$ in the clockwise direction at the
  endpoints of $P_0$ and the boundary curves of $\cA''$ intersect
  those of $\cD'$ exactly once. Further, $P_0\cap
  (I_n^1+(k+1)\omega)\neq\emptyset$.
\end{lem}

\begin{lem} \label{l.arc} There exists an arc
  $\Xi=\{(\theta,\xi(\theta))\mid \theta\in J+\omega\}\ssq
  (J+\omega)\times C$, with continuous $\xi:J+\omega\to C$, such that
  $P_1=\pi_1(\cA''\cap \Xi)$ is an interval.
\end{lem}
\begin{rem} \label{r.P_1} Note that since $\cD'\ssq (J+\omega)\times
  E$, $|\cA_\theta''|\leq \alpha^{-\frac{M_n}{p}+(p+\frac{1}{p})k}$
  and $d(C,E)>\alpha^{-\frac{M_n}{p}+(p+\frac{1}{p})k}$ if $n\in\N$ is
  sufficiently large, we have that $P_1\ssq P_0$. From now on, we
  always assume that this is the case.  Moreover, as the slope of
  $\partial_\theta\varphi_{\cA''}^\pm(\theta)$ is smaller than
  $\alpha^{pk}\cdot\frac{2S}{1-\alpha^{-p}}$, we obtain
  \begin{equation}
    \label{eq:13}
    d(P_1,\kreis\smin P_0) \ \geq \
    \frac{d(C,E)-2\alpha^{-\frac{M_n}{p}+(p+\frac{1}{p})k}}{\alpha^{p(k+1)}\cdot\frac{2S}{\alpha^p-1}}
    \ \geq \ d(C,E)\cdot \alpha^{-p(k+1)}/2 \ \geq \ \alpha^{-2pk} \ ,
  \end{equation}
  where the last inequality again requires that $n$ (and thus $k\geq
  M_{n-1}$) is sufficiently large.
\end{rem}

Denote by $J^-$ and $J^+$ the left, respectively right component of
$(J+\omega)\smin P_0$. Define
\begin{eqnarray*}
  \cL_1 & = & \{(\theta,x)\mid \theta\in J^-, x\in[e^-,\varphi^+_{\cD'}(\theta)]\} \\
  \cR_1 & = & \{(\theta,x)\mid \theta\in J^+, x\in[\varphi^-_{\cD'}(\theta),e^+]\} \ .
\end{eqnarray*}
Further, denote by $P_0^-$ and $P_0^+$ the left, respectively right
component of $P_0\smin \inte(P_1)$. Define
\begin{eqnarray*}
  \cL_2 & = & \{(\theta,x)\mid \theta\in P_0^-, x\in[e^-,\varphi^-_{\cD'}(\theta)]\} \\
  \cR_2 & = & \{(\theta,x)\mid \theta\in P_0^+, x\in[\varphi^+_{\cD'}(\theta),e^+]\} \ .
\end{eqnarray*}
Let $\cL=\cL_1\cup \cL_2$ and $\cR=\cR_1\cup\cR_2$. (See Figure~\ref{f.mainproof} for
an illustration.)
\begin{rem} \label{r.A''_inclusion} We have $\cA''\cap \left( \
    (J+\omega)\ \times E\right) \ssq \cL\cup\cR$ if $\alpha$ and $n$
  are large.
\end{rem}
\begin{proof}
  Note that since $\Xi\ssq (J+\omega) \times C$ and
  $|\cA''_\theta|\leq\alpha^{-\frac{M_n}{p}+(p+\frac{1}{p})k}<d(C,E)$
  (using $k\ll M_n$ and assuming $\alpha$ to be large), we have
  $\cA''\cap (P_1\times E)=\emptyset$. Moreover, $\cA''$ only crosses
  the arc $\Xi$ once. The statement therefore follows from the way
  $\cA''$ leaves and enters $\cD'\ssq (J+\omega) \times E$, according
  to Lemma~\ref{l.P_0}.
\end{proof}

Finally, the following statement ensures that at the extremal points of
$\Gamma$, different situations occur.
\begin{lem} \label{l.extremal_positions} At $\tau=\tau^-$, the set
  $\cB''$ intersects $\cR$, whereas at
  $\tau=\tau^+$ it intersects $\cL$.
\end{lem}

Based on the above statements, we can now turn to the proof of
Theorem~\ref{t.modelocking_quantitative}, which is illustrated in
Figure~\ref{f.mainproof}.
   \begin{figure}[t!]
  \begin{center}
 \includegraphics[width=0.6\linewidth]{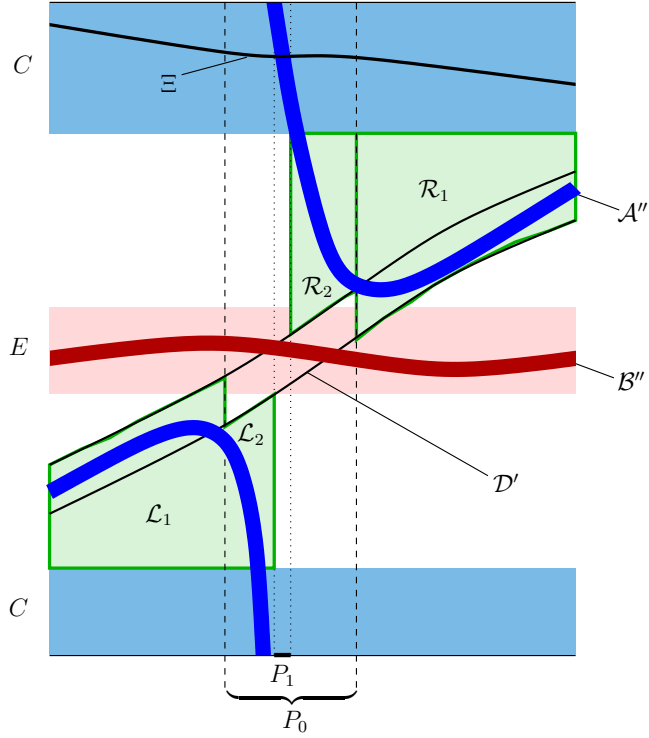}
  \end{center}

\caption{\small Illustration of the different sets
  considered in the proof of Theorem~\ref{t.modelocking_quantitative}.
  \label{f.mainproof}}

\end{figure}
\proof[\bf Proof of Theorem~\ref{t.modelocking_quantitative}.]

Suppose $\omega\in\cD(\gamma,\nu)$ and $\hat f\in\cP_\omega$ satisfies
conditions (\ref{eq:Cinvariance})--(\ref{eq:boundinterval}). Fix
$\delta>0$, denote by $\alpha_*',\eps_*'>0$ the constants given by
Theorem~\ref{t.Main_qualitative} and recall that $q\geq
\max\{8,4\nu\}$. We have to show that for any parameter $\tau_0$
contained in the set $\Lambda^{\hat f}$ from
Theorem~\ref{t.Main_qualitative} and any $\zeta>0$ there exists some
$\tau\in B_\zeta(\tau_0)$ such that $f_\tau$ is mode-locked, provided
that $\alpha>\alpha'_*$ is sufficiently large and $\eps_0<\eps'_*$ is
sufficiently small. 
 The largeness and smallness assumptions on $\alpha$ and $\eps_0$ will
 be used implicitely from now on, and all estimates below should be
 understood under this premise.

 To that end, fix $\tau_0\in\Lambda^{\hat f}$ and $\zeta>0$. Choose
 $n\in\N$ and $\Gamma=[\tau^-,\tau^+]\ssq B_\zeta(\tau_0)$ according
 to Lemma~\ref{l.perturbation_closereturn}. By
 Proposition~\ref{p.modelocking_criterion}, it suffices to find some
 $\tau\in\Gamma$ such that $\cC_n=\emptyset$. Moreover, due to
 Lemma~\ref{l.inclusions}, this follows if we can show that
 \begin{equation}
   \label{eq:9}
   f^{M_n}_\tau(\cA') \cap f_\tau^{-(M_n-k)}(\cB') \ = \ \cA''\cap\cB'' \ = \ \emptyset \ .
 \end{equation}
 Due to Lemma~\ref{l.preimage_shape}, we know that $\cB''\ssq
 (J+\omega)\times E$, and we have $\cA''\cap((J+\omega)\times E) \ssq
 \cL(\tau)\cup\cR(\tau)$ by Remark~\ref{r.A''_inclusion}. We claim
 that for all $\tau\in\Gamma$ the strip $\cB''$ can intersect at most
 one of the two sets $\cL$ and $\cR$.  Since it intersects $\cR$ for
 $\tau=\tau^-$ and $\cL$ for $\tau=\tau^+$ by
 Lemma~\ref{l.extremal_positions} and all sets are closed and depend
 continuously on $\tau$, this means that for some
 $\tau\in(\tau^-,\tau^+)$ we must have $f_\tau^{-(M_n-k)}(\cB')\cap
 (\cL(\tau)\cup\cR(\tau))=\emptyset$. This, in turn, yields
 (\ref{eq:9}) and thus completes the proof.\smallskip

 Hence, it remains to show that $\cB''$ cannot intersect both $\cL$
 and $\cR$ at the same time.  Suppose for a contradiction that
 $(\theta_1,x_1)\in \cB''\cap \cL$ and $(\theta_2,x_2)\in
 \cB''\cap\cR$.  By Lemma~\ref{l.preimage_shape}, we have
   \begin{equation}
           \label{e.x-value_difference}
           x_2-x_1 \ \leq \ \frac{S}{\alpha^{1/p}-1}\cdot|\theta_2-\theta_1|
             + |E|\cdot \alpha^{-\frac{M_n-k}{p}} \ .
   \end{equation}
   We distinguish four cases. Thereby, we will use freely the fact
   that $\alpha$ is sufficiently large and indicate when this is used
   by placing $(\alpha)$ over the respective inequality signs.
   \begin{description}
   \item[Case I] Suppose $(\theta_1,x_1)\in\cL_1$ and
     $(\theta_2,x_2)\in\cR_1$. In this case, we have that
     \begin{equation}
       \label{eq:12}
       \theta_2-\theta_1 \ \geq \ |P_0| \ \geq \ \frac{1-|C|}{4S}\cdot\alpha^{-pM_{n-1}}
     \end{equation}
    by Lemma~\ref{l.P_0} and
    \begin{eqnarray*}
      x_2-x_1 & \geq & \varphi^-_{\cD'}(\theta_2)-\varphi^+_{\cD'}(\theta_1)
      \\ &  \stackrel{\textrm{Lem.~\ref{l.D_shape}(ii)}}{\geq} & \
    \varphi^-_{\cD'}(\theta_2)-\varphi^-_{\cD'}(\theta_1) - \alpha^{-\frac{k-M_{n-1}}{p}}
      \\
      &\stackrel{\textrm{Lem.~\ref{l.D_shape}(iii)}}{\geq} &  \frac{s}{2}\cdot
      (\theta_2-\theta_1)-\alpha^{-\frac{k-M_{n-1}}{p}} \\ &
      \ \stackrel{(\ref{eq:12}),(\alpha)}{\geq} &
      \frac{S}{\alpha^{1/p}-1}\cdot(\theta_2-\theta_1) + \alpha^{-\frac{M_n-k}{p}} \ ,
    \end{eqnarray*}
    contradicting (\ref{e.x-value_difference}).
  \item[Case II] Suppose $(\theta_1,x_1)\in\cL_1$ and
    $(\theta_2,x_2)\in\cR_2$. In this case, we have
    \begin{equation}
      \label{eq:11}
      \theta_2-\theta_1 \ \geq \ d(P_1,\kreis\smin P_0) \ \geq \alpha^{-2pk}
    \end{equation}
    by Remark~\ref{r.P_1}. Moreover, the definitions of $\cL_1(\tau)$
    and $\cR_2(\tau)$ imply that
    \begin{eqnarray*}
      x_2-x_1 & \geq & \varphi^+_{\cD'}(\theta_2)-\varphi^+_{\cD'}(\theta_1) \\
      &  \stackrel{\textrm{Lemma~\ref{l.D_shape}(iii)},(\alpha)}{\geq} & 
       \frac{s}{2}\cdot(\theta_2-\theta_1)
      \ \stackrel{(\ref{eq:11}),(\alpha)}{\geq} \
      \frac{S}{\alpha^{1/p}-1}\cdot(\theta_2-\theta_1) +\alpha^{-\frac{M_n-k}{p}} \ ,
    \end{eqnarray*}
    again contradicting (\ref{e.x-value_difference}).
   \item[Case III] The case $(\theta_1,x_1)\in\cL_2$ and
     $(\theta_2,x_2)\in\cR_1$ is symmetric to the preceeding one and
     can be treated in the same way.
   \item[Case IV] Finally, suppose $(\theta_1,x_1)\in\cL_2(\tau)$ and
     $(\theta_2,x_2)\in\cR_2(\tau)$. In this case
     \begin{eqnarray*}
       x_2-x_1 & \geq & \varphi^+_{\cD'}(\theta_2) - \varphi^-_{\cD'}(\theta_1) \\
       & \stackrel{\textrm{Lemma~\ref{l.D_shape}}(ii)}{\geq}&
       \varphi^+_{\cD'}(\theta_2) - \varphi^+_{\cD'}(\theta_1) + 
         |C|\cdot \alpha^{-p(k-M_{n-1})} \\
       & \stackrel{\textrm{Lemma~\ref{l.D_shape}}}{\geq}& 
                \frac{s}{2}\cdot(\theta_2-\theta_1)  + |C|\cdot
       \alpha^{-p(k-M_{n-1})}\\
       &\stackrel{(\alpha)}{\geq}&
       \frac{S}{\alpha^{1/p}-1}\cdot(\theta_2-\theta_1) +\alpha^{-\frac{M_n-k}{p}} \ ,
     \end{eqnarray*}
    contradicting (\ref{e.x-value_difference}) as before. \qed\medskip

   \end{description}

\section{Proofs of the geometric estimates}\label{MainProof}

Throughout the proofs of this section, we will at most times omit
the parameter $\tau$ from the notation and write $f$, $f_{\theta}$
and $I^\iota_n$ instead of $f_\tau$, $f_{\tau,\theta}$ and
$I^\iota_n(\tau)$ (with the exception of the proof of
Lemma~\ref{l.perturbation_closereturn}). Moreover, we will always
assume implicitely that the parameter $\alpha$ is sufficiently large
and $\eps_0$ is sufficiently small. All estimates below should be
understood in this sense. Sometimes, but not always, we will
indicate that this fact is used by placing $(\alpha)$ or $(\eps_0)$
over the respective inequality signs.

\subsection{Proof of Lemma~\ref{l.perturbation_closereturn}. }
\label{CloseReturn}

According to (\ref{e.M_j}) and (\ref{e.eps_j}), there exists some $n\in \N$,
$n\geq 10$, so that $\frac{s}{4L}\varepsilon_{n-1}<\zeta$. We fix this $n$ and
first prove that there exists some $k\in\left[2K_{n-1}M_{n-1}+1,\
  M_{n-1}^{4q(\nu+1)}\right]$, such that
\begin{equation} \label{e.close_return_creation1}
d\left(I_n^1(\tau_0)+k\omega,I_n^2(\tau_0)\right) \ < \
 \frac{s\ell }{8LS}\varepsilon_{n-1}
\end{equation}
and $I_n^1(\tau_0)+k\omega$ is to the left of $I_n^2(\tau_0)$ in a
local sense. Since $(\cX)_{n-1}$ holds for $\tau_0$, it is obvious
that $k\geq 2K_{n-1}M_{n-1}+1$.

For any $N\in \N$, there is a positive integer $m\leq N$ such that $d(m\omega,
0)\leq \frac{1}{N}$. Moreover, since $\omega$ is Diophantine, we have
$d(m\omega, 0)\geq \gamma m^{-\nu}\geq \gamma N^{-\nu}$. Together, this implies
that after $N\left([\gamma^{-1}N^\nu]+1\right)$ iterates, the orbit of $\omega$
is $\frac{1}{N}$ dense in the circle. Thus, there exists some $k\leq
N\left([\gamma^{-1}N^\nu]+1\right)\leq 2\gamma^{-1}N^{\nu+1}$ such that
$d\left(I_n^1(\tau_0)+k\omega, I_n^2(\tau_0)\right)\leq 1/N$ and
$I_n^1(\tau_0)+k\omega$ is to the left of $I_n^2(\tau_0)$. Taking
$N=\left[\frac{4LS}{\ell}\alpha^{M_{n-2}/p}\right]+1$, we obtain
$$d(I_n^1(\tau_0)+k\omega, I_n^2(\tau_0))\ < \ \frac{\ell}{4LS}\alpha^{-M_{n-2}/p}\ 
\stackrel{(\ref{e.eps_j})}{\leq} \
\frac{s\ell }{8LS}\varepsilon_{n-1} \ $$ and
\begin{eqnarray*}
  k&\leq& \frac{2}{\gamma}\left(\frac{4LS}{\ell}\alpha^{M_{n-2}/p}+1\right)^{\nu+1} \\ 
& \stackrel{(\ref{e.M_j})}{\leq} & 
  \frac{2}{\gamma} \left(\frac{4LS}{\ell}M_{n-1}^{2q}+1\right)^{\nu+1} \  \stackrel{(\alpha)}{\leq} \
  M_{n-1}^{4q(\nu+1)} \ \stackrel{(\alpha)}{\ll}  \ M_n \ .
\end{eqnarray*}
We now claim that conditions $(\cX)_{n-1}$ and $(\cY'')_n$ are
satisfied for all $\tau$ with
$|\tau-\tau_0|<\frac{s}{4L}\varepsilon_{n-1}$. In order to do so, we
proceed by induction on $j$. Suppose that $(\mathcal{X})_{j-1}$ and
$(\mathcal{Y}'')_{j-1}$ hold for $\tau \in
B_{\frac{s}{4L}\eps_{n-1}}(\tau_0)$. Then we have
$|I_{j}^{\iota}(\tau)|\leq \varepsilon_j$ and $|\partial_\tau
I_j^\iota(\tau)|\leq \frac{2L}{s}$, $\iota=1,2,$ by Proposition
\ref{finite_times} if $j\geq 1$ and by assumption if $j=0$. If $d_H$
denotes the Hausdorff distance, then this implies that
$d_H(I_j^\iota(\tau),I_j^\iota(\tau_0)\leq \frac{2L}{s}\cdot
|\tau-\tau_0|$.

Thus, using $(\mathcal{X}')_{j}$ for $\tau_0$, we have for all
$l=1\ld 2K_jM_j$ and $\iota_1,\iota_2=1,2,$
\begin{eqnarray*}
  \lefteqn{d(I_j^{\iota_1}(\tau),I_j^{\iota_2}(\tau)+l\omega) }\\ & \geq &
  d(I_j^{\iota_1}(\tau_0),I_j^{\iota_2}(\tau_0)+l\omega)-d_H(I_j^{\iota_1}(\tau_0),I_j^{\iota_1}(\tau)) -
  d_H(I_j^{\iota_2}(\tau_0)+l\omega,I_j^{\iota_2}(\tau)+l\omega)\\
  &>&9\varepsilon_j-\frac{2L}{s}\cdot \frac{s\varepsilon_{n-1}}{4L}-\frac{2L}{s}\cdot
  \frac{s\varepsilon_{n-1}}{4L}  \ > \ 3\varepsilon_j \
\end{eqnarray*}
if $j\leq n-1$. Hence, $(\mathcal{X})_j$  holds for $\tau$. In a
similar way, using $(\cY')_j$ for $\tau_0$, we have that for
$\iota_1,\iota_2=1,2,\ j\leq n,\ 0\leq j' \leq j-1,$ and
$-M_{j'}\leq l\leq M_{j'}+2$
\[
  d(I_j^{\iota_1}(\tau)-(M_j-1)\omega,I_{j'}^{\iota_2}(\tau)+l\omega) \ > \
   \varepsilon_{j-1} \ ,
\]
and
\[
d(I_j^{\iota_1}(\tau)+(M_j+1)\omega,I_{j'}^{\iota_2}(\tau)+l\omega)\
> \ \varepsilon_{j-1} \ ,
\]
as long as $j\leq n$. Therefore, conditions $(\mathcal{X})_{n-1}$
and $(\mathcal{Y}'')_{n}$ hold for $\tau\in
B_{\frac{s}{4L}\eps_{n-1}}(\tau_0)$ as claimed.

Now, due to (\ref{est_relative}) the interval $I^1_n(\tau)+k\omega$ moves with
positive minimal speed $\ell/S$ relative to $I^2_n(\tau)$. Hence, it follows
from (\ref{e.close_return_creation1}) that $I^1_n(\tau)+k\omega$ moves from the
left to the right of $I_n^2(\tau)$ when $\tau$ transverses the interval
$[\tau_0,\tau_0+\frac{s}{4L}\eps_{n-1}]$. (note that $|I^\iota_n(\tau)|\leq
\eps_n$ by Proposition~\ref{finite_times} and $\eps_n\ll \eps_{n-1}$). Thus, we
can choose a subinterval $\Gamma=[\tau^-,\tau^+]\ssq
[\tau_0,\tau_0+\frac{s}{4L}\eps_{n-1}]$ that satisfies all the assertions of the
lemma.\qed\medskip

\subsection{Proof of Lemma~\ref{l.inclusions} }

We first need the following statement.
\begin{claim}\label{l.interval_disjoint}
  Let $\Gamma$ be as in Lemma~\ref{l.perturbation_closereturn}. Then for all
  $\tau\in\Gamma$ the following statement holds.
\begin{equation}\label{est_disjoint}
(J+l\omega)\cap \mathcal I_n=\emptyset,\ \ \ \ \forall\ \
l\in\{-M_n-k+1,\ldots, M_n+1\}\setminus \{-k,0\}.
\end{equation}
\end{claim}
\proof By Lemma \ref{l.perturbation_closereturn}, we have $k\leq
M_{n-1}^{4q(\nu+1)}\ll M_n$ for $\alpha$ large, and conditions $(\cX)_{n-1},
(\cY'')_{n-1}$ hold for all $\tau\in\Gamma$. Proposition~\ref{finite_times},
implies $|I_n^\iota|\leq \varepsilon_n,\ \iota=1,2$. Suppose
$l\in\{-M_n-k+1,\ldots, M_n+1\}\setminus\{-k, 0\}$. Then $\omega\in
\cD(\gamma,\nu)$, (\ref{e.M_j}) and (\ref{e.eps_j}) yield
\begin{eqnarray*}
\lefteqn{d(I_n^1+(k+l)\omega, I_n^1) \geq d((k+l)\omega, 0)-|I_n^1| } \\
 & \geq &    
\gamma\cdot |k+l|^{-\nu}-\varepsilon_n \ > \ \gamma\cdot
(2M_n)^{-\nu}-\varepsilon_n\ \stackrel{(\alpha)}{>}\ 9\varepsilon_n \ 
\end{eqnarray*}
and
\begin{eqnarray*}
\lefteqn{d(I_n^2+l\omega, I_n^2)\geq d(l\omega, 0)-|I_n^2| } \\  & \geq&  \gamma\cdot
|l|^{-\nu}-\varepsilon_n \ > \ \gamma\cdot
(2M_n)^{-\nu}-\varepsilon_n \  \stackrel{(\alpha)}{>}\ 9\varepsilon_n \ .
\end{eqnarray*}
Since $J\subseteq
B_{9\varepsilon_n}(I_n^1+k\omega)\cap B_{9\varepsilon_n}(I_n^2)$, this
implies (\ref{est_disjoint}). \qed\medskip

Now we can turn to the proof of Lemma~\ref{l.inclusions}. Since $I_n^2\cup
(I_n^1+k\omega)\ssq J$, we have that $\mathcal A_n^2\subseteq \mathcal A'$ and
$\mathcal B_n^1\subseteq \mathcal B'$ by definition. Thus, it will be sufficient
to show that
$$f^{-k}(\mathcal B_n^2)\ \subseteq \ \mathcal B'\ \ \textrm{and}\ \
f^k(\mathcal A_n^1) \ \subseteq \ \mathcal A'.$$

As for $\mathcal B_n^2$, we have $\left(I_n^2+(M_n+1)\omega\right)\cap \mathcal
V_{n-1}^+=\emptyset$ by condition $(\cY'')_{n}$. Moreover, $(\cY'')_n$ together
with $d(I_n^1+(M_n+1)\omega, I_n^2+(M_n-k+1)\omega)\leq 4\varepsilon_n\ll
\eps_{n-1}$ yields $(I_n^2+(M_n-k+1)\omega)\cap \mathcal
W_{n-1}^+=\emptyset$. Moreover, due to Claim~\ref{l.interval_disjoint} we have
that $k$ is the first integer such that $J-k\omega$ intersects $\cI_n$. Thus, we
can apply Lemma \ref{lem-ess_back} to obtain $f^{-k}(\mathcal B_n^2)\subseteq
\mathcal B'$.

Similarly, we have $\left(I_n^1-(M_n-1)\omega\right)\cap \mathcal
V_{n-1}^-=\emptyset$ by $(\cY'')_n$, and the fact that $d(I_n^2-(M_n-1)\omega,
I_n^1-(M_n-k-1)\omega)\leq 4\varepsilon_n\ll\eps_{n-1}$ together with $(\cY'')_n$
also imply $(I_n^1-(M_n-k-1)\omega)\cap \mathcal W^+_{n-1}=\emptyset$. Thus
Claim~\ref{l.interval_disjoint} combined with Lemma \ref{lem-ess} yield
$f^k(\mathcal A_n^1)\subseteq \mathcal A'$. \qed\medskip

\subsection{Proof of Lemma~\ref{l.preimage_shape} }

For the proof, we first need the following statement.

\begin{claim} \label{c.strict_nestedness}
Let $\hat f$ satisfy the assertion of Theorem
\ref{t.Main_qualitative} and assume $\tau\in \Lambda_{n-1}^{\hat
f}$. Then for $n\geq 3$, we have
\begin{equation}\label{claimresult-1}
f^{M_{n-1}}(\mathcal A_{n-1}^\iota)\subseteq \T^1\times E,\ \
\iota=1,2,
\end{equation}
\begin{equation}\label{claimresult-2}
 B_{9\varepsilon_n}(I_n^\iota)\subseteq I_{n-1}^\iota,\
\ \iota=1,2.
\end{equation}
\end{claim}
\begin{proof} The proof is illustrated in Figure~\ref{f.claim_inclusion_E}. We
  let
  \[
  I_j^\iota=(a_j^\iota, b_j^\iota), \ \  \tilde
  \cA_j^\iota=f^{M_j}(\cA_j^\iota),\ \ \hat
  \cB_j^\iota=f^{-M_j}(\cB_j^\iota),\ \ \iota=1,2,\ j=0,1,\ldots, n-1.
  \]
  We have $\cI_{j}-(M_{j}-1)\omega \cap \cV^-_{j-1} =
  \emptyset$ and $\cI_{j}-(M_{j-1}-1)\omega\cap \cW^+_{j-1}=\emptyset$ for all
  $j=0\ld n-1$ by $(\cY)_{n-1}$. Hence, Lemma~\ref{lem-ess} yields
  $f^{M_{j}-M_{j-1}}(\cA^\iota_{j})\ssq\cA^\iota_{j-1}$ for all $j=0\ld n-1$
  and therefore $\tilde\cA^\iota_{n-1}\ssq
 \tilde\cA^\iota_{n-2}\ssq \ldots \ssq \tilde\cA_2^\iota$. Thus, it
  suffices to prove (\ref{claimresult-1}) for the case $n=3$.
   \begin{figure}[ht!]
  \begin{center}
 \includegraphics[height=0.7\linewidth]{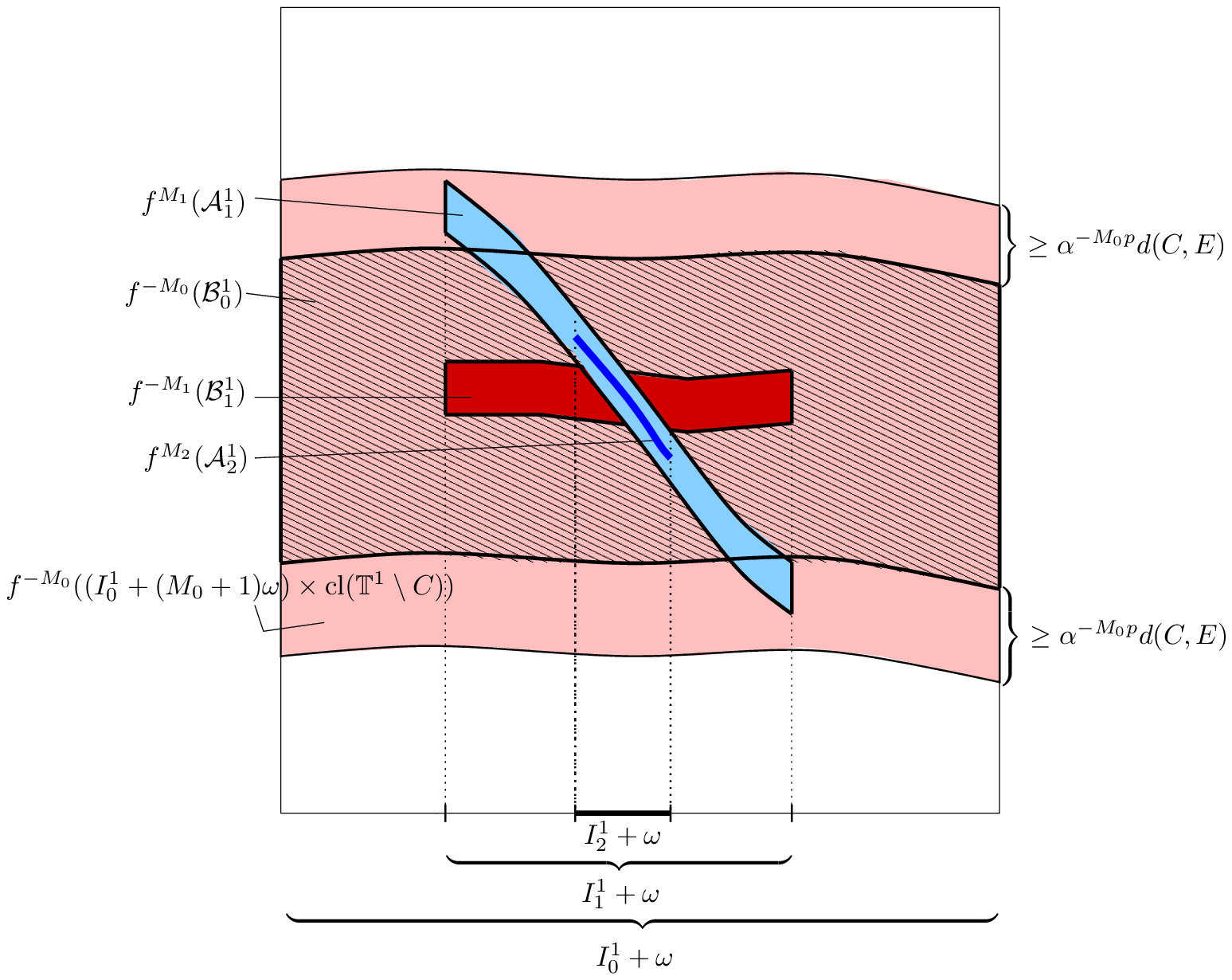}
  \end{center}

\caption{\small  Proof of the inclusion \eqref{claimresult-1} in Claim \ref{c.strict_nestedness}: Position of $f^{M_2}(\cA_2^1)$. Since this
  set is close to $f^{-M_1}(\cB^1_1)$, which lies well inside the expanding region, 
 we obtain $f^{M_2}(\cA_2^1)\ssq \kreis\times E$ as well. 
 \label{f.claim_inclusion_E}}

\end{figure}
A similar argument with Lemma~\ref{lem-ess_back} for the backwards iteration
yields $f^{-(M_1-M_0)}(\cB^\iota_1)\ssq \cB^\iota_0$.  Since
$f^{-M_0}(\cB^\iota_0) \ssq
f^{-M_0}((I^\iota_0+(M_0+1)\omega)\times\textrm{cl}(\kreis\smin C)) \ssq
\kreis\times E$ by (\ref{eq:Cinvariance}) and $(\cX)_0$, we can use
(\ref{eq:bounds1}) to obtain
\[
   d(\kreis\smin E,\hat \cB_{1,\theta}^\iota) \ \geq \ \alpha^{-M_0p} d(C,E)
\]
for all $\theta\in I^\iota_1+\omega$ (see Figure \ref{f.claim_inclusion_E}).
Moreover, by Lemma~\ref{estimate.forward}, we have that
\[
|\varphi_{\tilde{\mathcal
A}_1^\iota}^+(\theta)-\varphi_{\tilde{\mathcal
  A}_1^\iota}^-(\theta)| \ \leq \ \alpha^{-M_1/p}|C| \
\stackrel{(\alpha)}{<} \ \alpha^{-M_0p}d(E,C) \ .
\]
Therefore, by definition of
$I^\iota_2=\inte\left(\pi_1\left(\hat\cB^\iota_1\cap\tilde\cA^\iota_1\right)\right)$,
this yields
$$f^{M_2}(\mathcal
A_2^\iota) \ \ssq \ f^{M_1}\left((I_2^\iota-(M_1-1)\omega)\times
  C\right)\subseteq \T^1\times E \ .$$ This proves (\ref{claimresult-1}).

As for (\ref{claimresult-2}), we first consider the case $\iota=1$ and verify
that ${\tilde \cA}_{j}^1$ `crosses' ${\hat \cB}_{j}^{1}$ `downwards' for all
$j=0,1,\ldots,n-1$. Since $( I_j^1-(M_j-1)\omega)\cap\cV_{j-1}^-=\emptyset,
I_j^1\ssq \cI_{j-1}$ and $\forall \ l=0,1,\ldots, M_j-2, (
I_j^1-(M_j-1-l)\omega)\cap \cI_j=\emptyset$ for all $j=0,1,\ldots, n-1$ by
$(\cX)_{n-1}, (\cY)_{n-1}$, Lemma \ref{estimate.forward} implies that 
\[
-S-\frac{S}{\alpha^{1/p}-1}\ \leq\ \partial_\theta \varphi_{\tilde
{\cA}_j^1}^\pm(\theta) \ < \ -s+\frac{S}{\alpha^{1/p}-1}.
\]
A similar argument with Lemma \ref{estimate.back} for the backwards iteration
yields
\[
|\partial_\theta\varphi_{\hat{\cB}_j^1}^\pm(\theta)|\leq
\frac{S}{\alpha^{1/p}-1}, \ \ \forall \ j=0, 1,\ldots, n-1.
\]
Therefore, for any $\iota_1,\iota_2\in\{\pm\}, j=0,1,\ldots, n-1$,
we have
\begin{itemize}
\item[(i)] \  $-2S\stackrel{(\alpha)}{<}-S-\frac{2S}{\alpha^{1/p}-1}\leq
\partial_\theta(\varphi_{\tilde{\cA}_j^1}^{\iota_1}(\theta)-\varphi_{\hat{\cB}_{j}^1}^{\iota_2}(\theta))
\leq -s+\frac{2S}{\alpha^{1/p}-1}\stackrel{(\alpha)}{<}-s/2.$
\end{itemize}
Moreover, $\tilde{\cA}_0^1$ is above $\hat{\cB}_{0}^1$ at the left end of
$I_0^1+\omega$ and below at the right end by (\ref{eq:crossing}) and
(\ref{eq:s}), since $\tilde \cA_0^1\ssq f(I_0^1\times C)$ and $\hat \cB_0^1\ssq
(I_0^1+\omega)\times E$ by (\ref{eq:Cinvariance}) and $(\cX)_0$. Since $\tilde
\cA_{n-1}^1\ssq \cdots \ssq \tilde \cA_0^1$ and $\hat \cB_{n-1}^1\ssq \cdots\ssq
\hat \cB_0^1$ by $(\cX)_{n-1}, (\cY)_{n-1}$ and Lemma \ref{lem-ess},
\ref{lem-ess_back}, the definition of $I_j^1$ yields
\begin{itemize}
\item[(ii)]   $\tilde{\cA}_j^1$ is above $\hat{\cB}_{j}^1$ at
the left end of $I_j^1+\omega$ and below at the right end. (see
Figure \ref{f.claim_cross})
\end{itemize}
Thus, (i) and (ii) ensure that $\tilde \cA_j^1$ `crosses' $\hat\cB_j^1$
`downwards' (and give a precise meaning to this statement).  Hence, by 
definition of $I_{j+1}^1$, we have
\begin{equation}\label{esti-claim-1}
\varphi_{\tilde{\mathcal
A}_{j}^1}^-(a_{j+1}^1+\omega)-\varphi_{\hat{\mathcal
B}_{j}^1}^+(a_{j+1}^1+\omega)=0,\ \ j=0,1,\ldots,n-1,
\end{equation}
\begin{equation}\label{esti-claim-2}
\varphi_{\tilde{\mathcal
A}_{j}^1}^+(b_{j+1}^1+\omega)-\varphi_{\hat{\mathcal
B}_{j}^1}^-(b_{j+1}^1+\omega)=0,\ \ j=0,1,\ldots,n-1 \ .
\end{equation}
Further, we have $(I_{n-1}^1-(M_{n-1}-1)\omega)\cap \cV_{n-2}^-=\emptyset$ and
$(I_{n-1}-M_{n-2}\omega)\cap \cW^+_{n-2}=\emptyset$ by $(\cY)_{n-1}$, so that
Lemma \ref{lem-ess} implies $f^{M_{n-1}-M_{n-2}-1}(\mathcal A_{n-1}^1)\subseteq
\T^1\times C$.  Moreover, since $( I_{n-1}^1-M_{n-2}\omega)\cap \mathcal
I_0=\emptyset$ by $(\cY)_{n-1}$, we also have $f((
I_{n-1}^1-M_{n-2}\omega)\times \textrm{cl}(\T^1\setminus E))\subseteq \T^1\times
C$. Thus, for $\theta\in I_{n-1}^1+\omega$,
$$|f_{\theta-M_{n-1}\omega}^{M_{n-1}-M_{n-2}}(c^\pm)-c^\pm|\geq \alpha^{-p}d(C,E).$$
Combined with (\ref{eq:bounds1}) this means that if $\theta\in
I_{n-1}^1+\omega$, then 
\begin{eqnarray*}
  \lefteqn{ |\varphi_{\tilde{\mathcal
        A}_{n-1}^1}^{\pm}(\theta)-\varphi_{\tilde{\mathcal
        A}_{n-2}^1}^{\pm}(\theta)|
    \ = \ |f_{\theta-M_{n-2}\omega}^{M_{n-2}}(f_{\theta-M_{n-1}\omega}^{M_{n-1}-M_{n-2}}(c^\pm))
    -f_{\theta-M_{n-2}\omega}^{M_{n-2}}(c^\pm)| } 
    \\  &\geq&
    \alpha^{-pM_{n-2}}|f_{\theta-M_{n-1}\omega}^{M_{n-1}-M_{n-2}}(c^\pm)-c^\pm| \ \geq \alpha^{-p(M_{n-2}+1)}d(C,E).
\end{eqnarray*}
Similarly, given $\theta\in I_{n-1}^1+\omega$ we have 
$$|\varphi_{\hat{\mathcal
B}_{n-1}^\iota}^{\pm}(\theta)-\varphi_{\hat{\mathcal
B}_{n-2}^\iota}^{\pm}(\theta)|\geq \alpha^{-p(M_{n-2}+1)}d(C,E).$$
This yields 
\begin{eqnarray}\label{esti-claim-3}
\nonumber\lefteqn{\varphi_{\tilde{\mathcal
A}_{n-1}^1}^-(a_{n-1}^1+\omega)-\varphi_{\hat{\mathcal
B}_{n-1}^1}^+(a_{n-1}^1+\omega)} \\ & = &\varphi_{\tilde{\mathcal
A}_{n-1}^1}^-(a_{n-1}^1+\omega)-\varphi_{\tilde{\mathcal
A}_{n-2}^1}^-(a_{n-1}^1+\omega) \\ \nonumber &  & + \ \varphi_{\hat{\mathcal
B}_{n-2}^1}^+(a_{n-1}^1+\omega)-\varphi_{\hat{\mathcal
B}_{n-1}^1}^+(a_{n-1}^1+\omega)\\ \nonumber &\geq&
2\alpha^{-p(M_{n-2}+1)}d(C,E),
\end{eqnarray}
(see Figure \ref{f.claim_cross} with $j=n$) and similarly
\begin{equation}\label{esti-claim-4}
\varphi_{\tilde{\mathcal
A}_{n-1}^1}^+(b_{n-1}^1+\omega)-\varphi_{\hat{\mathcal
B}_{n-1}^1}^-(b_{n-1}^1+\omega)\leq -2\alpha^{-p(M_{n-2}+1)}d(C,E).
\end{equation}
Thus, by (i), (\ref{esti-claim-1}) with $j=n-1$, (\ref{esti-claim-3}),
(\ref{e.M_j}) and (\ref{e.eps_j}), we obtain
$$a_n^1-a_{n-1}^1\ \geq\  \frac{d(C,E)}{S}\alpha^{-p(M_{n-2}+1)} \ \geq\  \alpha^{-2pM_{n-2}}
\stackrel{(\alpha)}{\geq}\ 
\varepsilon_{n-1}^{8p^2}\stackrel{(\alpha)}{>}9\varepsilon_n.$$
Similarly, (i), (\ref{esti-claim-2}) with $j=n-1$,
(\ref{esti-claim-4}), (\ref{e.M_j}) and (\ref{e.eps_j}) yield
$$b_{n-1}^1-b_n^1 \ > \ 9\varepsilon_n.$$
For the intervals $I_{n-1}^2$ and $I_n^2$ the situation is exactly the same,
except for the fact that $\tilde{\cA}_j^2$ crosses $\hat{\cB}_j^2$ upwards
instead of downwards.
\end{proof}
\medskip

 \begin{figure}[ht!]
  \begin{center}
 \includegraphics[height=0.7\linewidth]{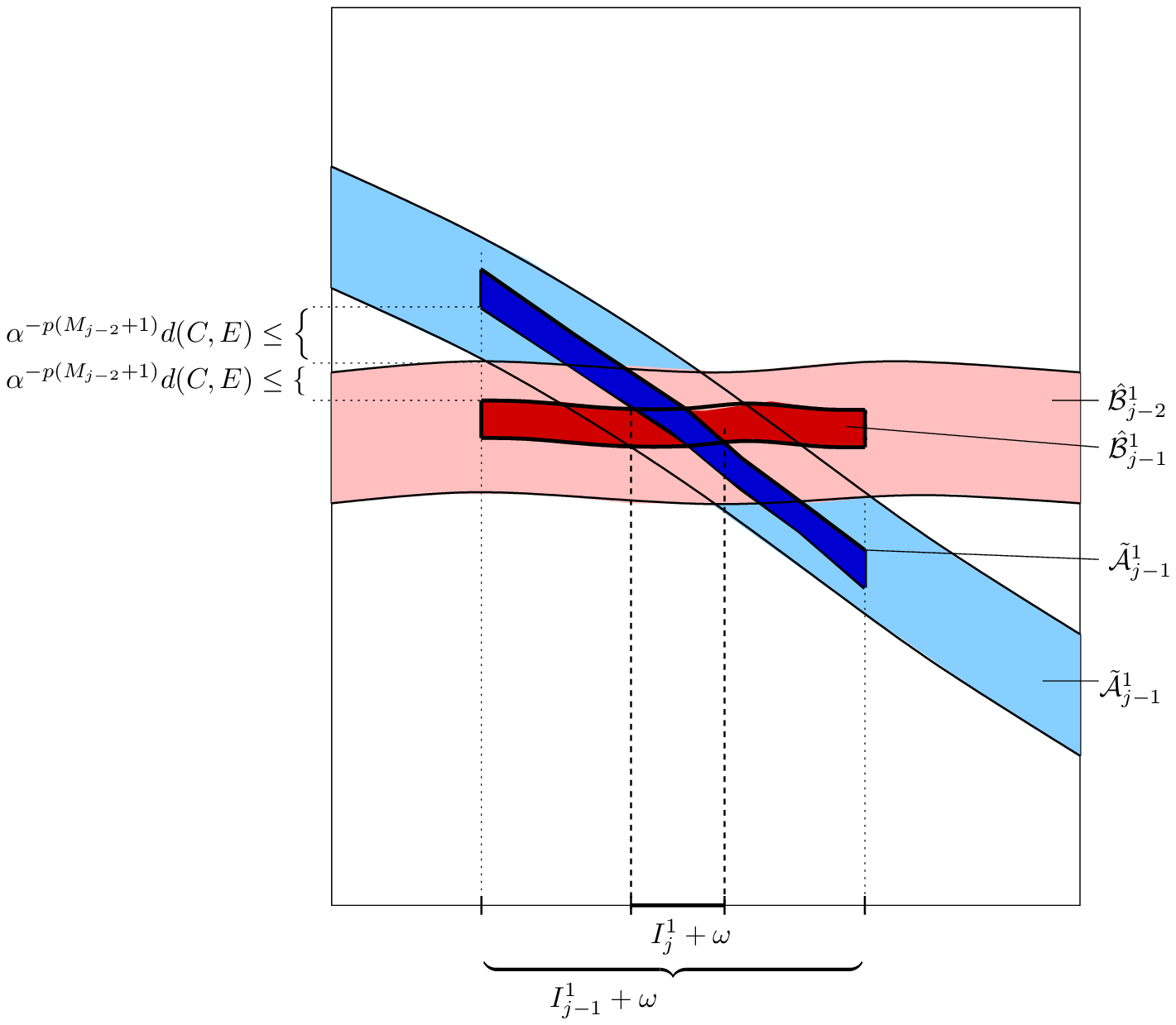}
  \end{center}

\caption{\small  Proof of Claim~\ref{c.strict_nestedness}: 
The `downwards' crossing between $\tilde \cA_j^1$ and
$\hat \cB_j^1$.
\label{f.claim_cross}}

\end{figure}%

We now turn to the proof of Lemma~\ref{l.preimage_shape}.  For $\tau\in \Gamma$,
by Lemma \ref{l.perturbation_closereturn} and Proposition \ref{finite_times}, we
have that $(\cX)_{n-1}, (\cY'')_n$ are satisfied and $|I_n^1(\tau)|,
|I_n^2(\tau)| \leq \varepsilon_n$.

Due to the definition of $J$, we have
\begin{equation}\label{equ-c-4-1}
J\subseteq B_{9\varepsilon_n}(I_n^1+k\omega)\cap
B_{9\varepsilon_n}(I_n^2),
\end{equation}
which implies $ J+(M_n-k+1)\omega\subseteq
B_{9\varepsilon_n}(I_n^1+(M_n+1)\omega)$. Since $9\varepsilon_n\leq
\varepsilon_{n-1}$, $(\cY'')_n$ yields
$$(J+(M_n-k+1)\omega)\cap \cV_{n-1}^+=\emptyset.$$ Condition
$(\cX)_{n-1}$ and (\ref{claimresult-2}), (\ref{equ-c-4-1}) also
imply that
$$(J+\omega)\cap \cW^-_{n-1}=\emptyset.$$ Therefore, by (\ref{est_disjoint}),  Lemma
\ref{lem-ess_back} implies that $\mathcal B''\subseteq (J+\omega)\times
E$. Moreover, since $(J+\omega)\subseteq I_{n-1}^2+\omega$ by
(\ref{claimresult-2}) and (\ref{equ-c-4-1}), and
$$\big(J+(M_n-k+1)\omega-l\omega\big)\cap (\mathcal I_n+\omega)=\emptyset\ \ \textrm{for}\ \ 0\leq l\leq M_n-k-1,$$
by (\ref{est_disjoint}), we can apply Lemma \ref{estimate.back} with $N=M_n-k$
and $\phi(\theta)=e^\pm$ to obtain that for any $\theta\in J+\omega$,
$$|\mathcal B_\theta''|\ \leq \ |E| \alpha^{-\frac{M_n-k}{p}},\
\textrm{and}\ \
$$
$$|\partial_\theta \varphi_{\mathcal B''}^\pm(\theta)|\ \leq \ \sum_{l=1}^{M_n-k}
\alpha^{-\frac{l}{p}}S\leq \frac{1}{\alpha^{1/p}-1}S.$$ \qed\medskip

\subsection{Proof of Lemma \ref{l.A'image_shape} }\label{s.A'}

For $\tau\in \Gamma$, by Lemma \ref{l.perturbation_closereturn}, we
have that $(\cX)_{n-1}$, $(\cY'')_n$ hold and $d(I_n^1+k\omega,
I_n^2)\leq 4\varepsilon_n$. Let $\mathcal A'''=f^{M_n-k}(\mathcal
A')$. Then for $\theta\in J+\omega$, we have $$\varphi_{\mathcal
A''}^\pm(\theta)=f_{\theta-k\omega}^k(\varphi_{\mathcal
A'''}^\pm(\theta-k\omega)).$$ We first derive the estimates for the
shape of $\mathcal A'''$. Since $(J-k\omega)\subseteq I_{n-1}^1$ by
(\ref{claimresult-2}) and (\ref{equ-c-4-1}),
$\big(J-(M_n-1)\omega+l\omega\big)\cap \mathcal{I}_{n}=\emptyset$
for $0\leq l\leq M_n-k-2$ by (\ref{est_disjoint}), and
$(J-(M_n-1)\omega)\cap \mathcal Y_{n-1}=\emptyset$ by $(\cY'')_n$
and (\ref{equ-c-4-1}). Therefore, we can apply Lemma \ref{estimate.forward} to obtain that 
\begin{equation}\label{esti-A'''-1}
 |\mathcal A_{\theta}'''|\leq \alpha^{-\frac{M_n-k}{p}}|C|,
\end{equation}
\begin{equation}\label{esti-A'''-2}
-S-\frac{S}{\alpha^{1/p}-1}  \leq  \partial_\theta\varphi_{\mathcal
A'''}^\pm(\theta)\leq -s+\frac{S}{\alpha^{1/p}-1}
\end{equation}
for all $\theta\in J-(k-1)\omega$.

Now in order to obtain the required estimates on $\mathcal A''$, we
let $\phi^\iota(\theta-k\omega)=\varphi_{\mathcal
A'''}^\pm(\theta-k\omega)$ $(\iota=1,2)$  for $\theta\in J+\omega$.
Then Remark \ref{esti-comm} yields
$$|\mathcal A_\theta''|\ \leq \ \alpha^{pk}\alpha^{-\frac{M_n-k}{p}}|C|,$$
$$|\partial_\theta \varphi_{\mathcal A''}^\pm(\theta)|\ \leq\  \frac{2\alpha^{p(k+1)}S}{\alpha^p-1}.$$
\qed\medskip

\subsection{Proof of Lemma \ref{l.D_shape} }


Since $J-(k-M_{n-1}-1)\omega\subseteq I_{n-1}^1+(M_{n-1}+1)\omega$
and $J-(M_{n-1}-1)\omega\subseteq I_{n-1}^2-(M_{n-1}-1)\omega$ by
(\ref{claimresult-2}) and (\ref{equ-c-4-1}), conditions
$(\cX)_{n-1}$ and $(\cY'')_{n-1}$ imply that
\begin{equation}\label{equ-c-6-1}
(J-(k-M_{n-1}-1)\omega)\cap \mathcal V_{n-1}^-=\emptyset,
\end{equation} and
\begin{equation}
(J-(M_{n-1}-1)\omega)\cap \mathcal W^+_{n-1}=\emptyset.
\end{equation}
Then by (\ref{est_disjoint}) and Lemma \ref{lem-ess}, we have
$f^{k-2M_{n-1}}(\mathcal D)\subseteq (J-(M_{n-1}-1)\omega) \times C$. Combined
with (\ref{claimresult-1}), this yields $$\mathcal D'\subseteq
f^{M_{n-1}}(\mathcal A_{n-1}^2)\subseteq (I_{n-1}^2+\omega)\times E.$$

Because $(J-(k-M_{n-1}-1)\omega+l\omega)\cap \mathcal I_n=\emptyset$ for $0\leq
l\leq k-M_{n-1}-2$ by (\ref{est_disjoint}), and $J\subseteq I_{n-1}^2$ by
(\ref{claimresult-2}) and (\ref{equ-c-4-1}), we can apply Lemma
\ref{estimate.forward} with $I=J-(k-M_{n-1}-1)\omega,\ N=k-M_{n-1}-1,\
\phi^\iota=c^\pm$ $ (\iota=1,2)$, together with (\ref{equ-7}), to obtain that
for all $\theta\in J+\omega$
$$|C|\cdot\alpha^{-p(k-M_{n-1})}\ \leq\  |\mathcal D_\theta'|\ \leq\  |C|\cdot\alpha^{-\frac{k-M_{n-1}}{p}},$$
$$s-\frac{S}{\alpha^{1/p}-1}\ \leq  \ \partial_\theta\varphi_{\mathcal D'}^\pm(\theta)\
\leq\ S+\frac{S}{\alpha^{1/p}-1}.$$
\qed\medskip

\subsection{Proof of Lemma \ref{l.P_0} }\label{s.P_0}

As before, we fix $\tau\in\Gamma$ such that assertions $(i)$-$(iii)$ of Lemma
\ref{l.perturbation_closereturn} hold.

Since $\mathcal A'' \cap \mathcal D'=f^{k}\big(f^{M_n-k}(\mathcal
A')\cap f^{-M_{n-1}}( \mathcal D) \big)$, we have
$\pi_1\left(\cA''\cap \cD'\right)=\pi_1\big(f^{M_n-k}(\mathcal
A')\cap f^{-M_{n-1}}( \mathcal D) \big)+k\omega$. In the following,
we will focus on $f^{M_n-k}(\mathcal A')\cap f^{-M_{n-1}}( \mathcal
D)$.  We let $\mathcal A'''=f^{M_n-k}(\cA')$ as before and set
$$\mathcal D_\iota=\left(J-(k-M_{n-1}-1)\omega\right)\times
X_\iota,\ \ \
  \hat {\mathcal D}_\iota=f^{-M_{n-1}}(\mathcal
D_\iota),\ \ \iota=1,2,3,c,
$$
where $X_1=[c^+, e^-], X_2=E, X_3=[e^+, c^-]$ and $X_c=[c^+, c^-]$.
Note that $\varphi_{\hat
\cD_c}^\pm(\theta)=\varphi_{f^{-M_{n-1}}(\cD)}^\mp(\theta)$ for
$\theta\in J-(k-1)\omega$.

We will first prove that $\cA'''$ crosses $\hat{\cD}_c$ exactly once and this
crossing is downwards (see Figure \ref{f.P_0}).
   \begin{figure}[ht!]
  \begin{center}
 \includegraphics[height=0.6\linewidth]{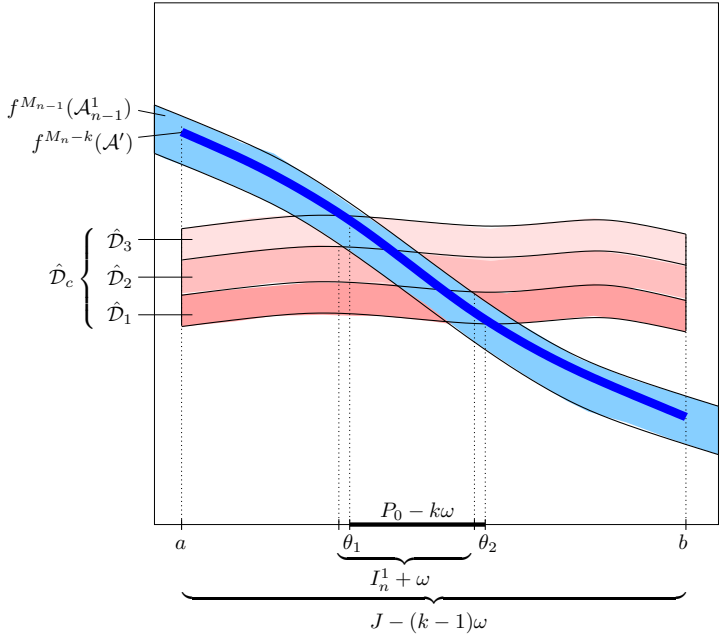}
  \end{center}

\caption{\small Proof of Lemma~\ref{l.P_0}: The definition of $P_0$.
  \label{f.P_0}}

\end{figure}
  The reason
is as follows. Since $J-(k+M_{n-1}-1)\omega\subseteq
I_{n-1}^1-(M_{n-1}-1)\omega$ by (\ref{claimresult-2}) and
(\ref{equ-c-4-1}),  and $(I_{n-1}^1-(M_{n-1}-1)\omega)\cap \mathcal
W^-_{n-1}=\emptyset$ by conditions $(\cX)_{n-1},\ (\cY'')_{n-1}$,  we
have $(J-(k+M_{n-1}-1)\omega)\cap \mathcal W^-_{n-1}=\emptyset$.
Together with (\ref{est_disjoint}) and the fact that
$(J-(M_n-1)\omega)\cap \mathcal Y_{n-1}=\emptyset$ by
(\ref{equ-c-4-1}) and $(\cY'')_{n}$, Lemma \ref{lem-ess} implies
$f^{M_n-k-M_{n-1}}(\mathcal A')\subseteq \mathcal A_{n-1}^1$ and
hence
$$\cA''' \ \subseteq \ f^{M_{n-1}}(\mathcal A_{n-1}^1).$$

Recall that $I_n^\iota=(a_n^\iota, b_n^\iota), \iota=1,2$.  By the
definition of $I_n^1$, we have
\begin{equation}\label{e.graph_equality}
\varphi_{\cA'''}^-(a_n^1+\omega)\geq
\varphi_{f^{M_{n-1}}(\cA_{n-1}^1)}^-(a_n^1+\omega) \ = \ \varphi_{\hat{\cD}_3}^-(a_n^1+\omega)\ .
\end{equation}
Since $J-(k-M_{n-1}-1)\omega\ssq I_{n-1}^1+(M_{n-1}+1)\omega$ and
$J-(k-1)\omega\ssq I_{n-1}^1+\omega$ by (\ref{claimresult-2}) and
(\ref{equ-c-4-1}) and $(J-(k-M_{n-1}-1)\omega-l\omega)\cap
(\cI_n+\omega)=\emptyset$ for $l=0,1,\ldots, M_{n-1}-1$ by (\ref{est_disjoint}),
we can apply Lemma \ref{estimate.back}, together with (\ref{equ-7}) to obtain
that
 $$|X_\iota|\cdot\alpha^{-pM_{n-1}}\ \leq\ |\hat{\cD}_{\iota,\theta}|\ \leq\ |X_\iota|\cdot\alpha^{-\frac{M_{n-1}}{p}} ,\ \ \
|\partial_\theta\varphi_{\hat{\cD}_\iota}^\pm(\theta)|\ \leq\
\frac{S}{\alpha^{1/p}-1},$$ 
for all $\theta\in J-(k-1)\omega$ and $\iota=1,2,3,c$.

Writing $J-(k-1)\omega=:[a,b]$ and using (\ref{esti-A'''-1}) and
(\ref{esti-A'''-2}), we obtain
\begin{eqnarray*}
\lefteqn{\varphi_{\cA'''}^-(a)-\varphi_{\hat{\cD}_3}^-(a)\ \stackrel{\eqref{e.graph_equality}}{\geq}\ 
 \varphi_{\cA'''}^-(a)-\varphi_{\cA'''}^-(a_n^1+\omega)
+\varphi_{\hat{\cD}_3}^-(a_n^1+\omega)-\varphi_{\hat{\cD}_3}^-(a)}\\
&\geq&
(s-\frac{S}{\alpha^{1/p}-1})\cdot(a_n^1+\omega-a)-\frac{S}{\alpha^{1/p}-1}\cdot(a_n^1+\omega-a)\hspace{4eM}\\
&\stackrel{(\alpha)}{\geq}&\frac{s}{2}\cdot 4\varepsilon_n\geq
4\alpha^{-\frac{M_{n-1}}{p}}>\sup_{\theta\in J-(k-1)\omega}
|\hat{\cD}_{3,\theta}|,
\end{eqnarray*}
 which means
$$\varphi_{\cA'''}^-(a)>\varphi_{\hat{\cD}_3}^+(a)=\varphi_{\hat{\cD}_c}^+(a).$$
Similarly, we obtain
$$\varphi_{\cA'''}^+(b)<\varphi_{\hat{\cD}_1}^-(b)=\varphi_{\hat{\cD}_c}^-(b).$$
Thus, together with the fact that $\inf_{\theta\in
  J-(k-1)\omega}|\partial_\theta\varphi_{\cA'''}^\pm(\theta)|>\sup_{\theta\in
  J-(k-1)\omega}|\partial_\theta\varphi_{\hat{\cD}_c}^\pm(\theta)|,$
we have that $\cA'''$ `downwards' crosses $\hat{\cD}_c$ exactly one
time, which means that in the image the boundary curves of $\cA''$
intersect those of $\cD'$ exactly once. Equivalently,
\begin{displaymath}
\left. \begin{array}{ll} &\exists !\  \theta_1\in J-(k-1)\omega \
\ \textrm{with}\
\varphi_{\cA'''}^+(\theta_1)=\varphi_{\hat{\cD}_c}^+(\theta_1) \
\textrm{and}\\
& \exists ! \ \theta_2\in J-(k-1)\omega \ \ \textrm{with}\
\varphi_{\cA'''}^-(\theta_2)=\varphi_{\hat{\cD}_c}^-(\theta_2).
\end{array}
\right.
\end{displaymath}
Then we have
\begin{eqnarray*}
\varphi_{\cA'''}^-(\theta_1)-\varphi_{\hat{\cD}_c}^-(\theta_1)&=&(\varphi_{\hat{\cD}_c}^+(\theta_1)-\varphi_{\hat{\cD}_c}^-(\theta_1))
-(\varphi_{\cA'''}^+(\theta_1)-\varphi_{\cA'''}^-(\theta_1))\\
&\geq& \inf_{\theta\in
J-(k-1)\omega}|\hat{\cD}_{c,\theta}|-\sup_{\theta\in
J-(k-1)\omega}|\cA_\theta'''|,
\end{eqnarray*}
$$\varphi_{\cA'''}^-(\theta_1)-\varphi_{\hat{\cD}_c}^-(\theta_1)\ \leq\  \sup_{\theta\in
J-(k-1)\omega}|\hat{\cD}_{c,\theta}|$$
$$s/2 \ <\
\partial_\theta(\varphi_{\hat{\cD}_c}^-(\theta)-\varphi_{\cA'''}^-(\theta))\ < \ 2S.$$
Therefore, for $\alpha$ large, we obtain that
\begin{equation}
\frac{(1-|C|)}{4S}\cdot\alpha^{-pM_{n-1}}\ <\ \theta_2-\theta_1\ <\
\frac{4(1-|C|)}{s}\cdot \alpha^{-M_{n-1}/p}.
\end{equation}
Moreover, if we let $\tilde\cA_{n-1}^1=f^{M_{n-1}}(\cA_{n-1}^1)$,
then by the definition of $I_{n}^1$, we have
$\varphi_{\tilde\cA_{n-1}^1}^+(b_n^1+\omega)=\varphi_{\hat{\cD}_2}^-(b_n^1+\omega)$.
Because
$\partial_\theta(\varphi_{\tilde\cA_{n-1}^1}^+-\varphi_{\hat{\cD}_2}^-)<-s/2<0$,
and
$$\varphi_{\tilde\cA_{n-1}^1}^+(\theta_1)-\varphi_{\hat{\cD}_2}^-(\theta_1)>\varphi_{\cA'''}^+(\theta_1)
-\varphi_{\hat{\cD}_2}^-(\theta_1)>\varphi_{\cA'''}^+(\theta_1)-\varphi_{\hat{\cD}_3}^+(\theta_1)=0,$$
we get $\theta_1<b_n^1+\omega$. Similarly, we obtain
$\theta_2>a_n^1+\omega$, which implies that $(\theta_1,\theta_2)\cap
(I_{n}^1+\omega)\neq\emptyset$.

If we now let $P_0=(\theta_1+k\omega, \theta_2+k\omega)$, then by the
selection of $\theta_1$ and $\theta_2$, we have
$$\pi_1(\cA''\cap \mathcal D' )\cap P_0=\emptyset,$$
with $\frac{(1-|C|)}{4S}\cdot\alpha^{-pM_{n-1}}< |P_0| <
\frac{4(1-|C|)}{s}\cdot \alpha^{-M_{n-1}/p}$ and $P_0\cap
(I_{n}^1+(k+1)\omega)\neq\emptyset$.
\qed \medskip

\subsection{Proof of Lemma \ref{l.arc} }

For $\theta\in I_0^2+\omega$, we let 
\[
\zeta(\theta) \ = \ c^-+\min\left\{\varphi_{f(I_0^2\times(\T^1\setminus
  E))}^-(\theta)-\varphi_{f(I_0^2\times (\T^1\setminus E))}^-(a_0^2+\omega),
c^+-c^-\right\} \ ,
\]
and choose $\xi$ to be a small $\mathcal C^1$-perturbation of $\zeta$ which
satisfies $\xi(\theta)\in C$ and $|\partial_\theta \xi(\theta)|\leq S$ by
(\ref{eq:bounddth}) (see Figure~\ref{f.P_1}).
   \begin{figure}[ht!]
  \begin{center}
 \includegraphics[height=0.6\linewidth]{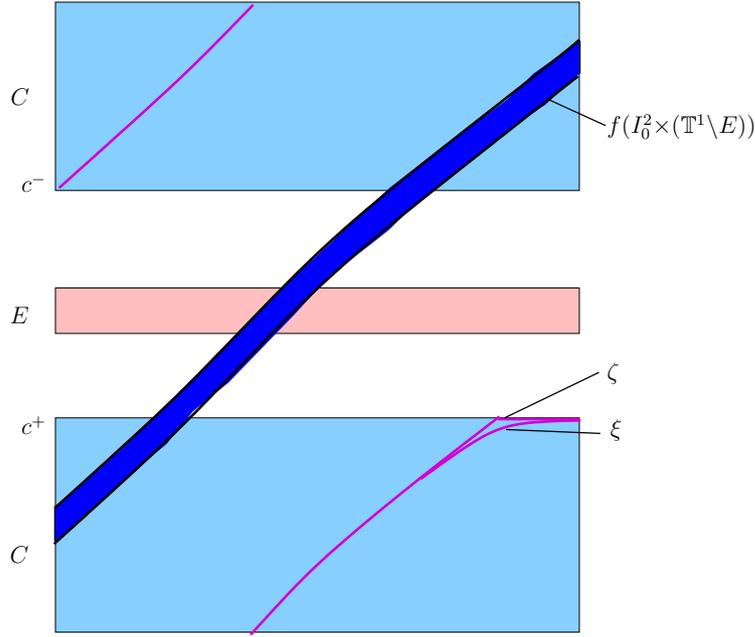}
  \end{center}

\caption{\small   Proof of Lemma~\ref{l.arc}: Construction of the curves $\zeta$ and $\xi$.
  \label{f.P_1}}

\end{figure}
Then, we let $\Xi=\{(\theta,\xi(\theta))\ |\ \theta\in J+\omega\}\subseteq
(J+\omega)\times C$. Since the graph of $\zeta$ is disjoint from
$f(J\times(\kreis\smin E))$, we can choose $\xi$ sufficiently close to $\zeta$
such that this still holds, that is, 
\begin{equation} \label{e.curve_disjointness}
\Xi\cap f(J\times (\T^1\setminus
E)) \ = \ \emptyset \ .
\end{equation}
In order to verify that $\pi_1(\Xi\cap \cA'')$ is an iterval, we consider the
preimage $f^{-k}(\Xi)$ and show that this curve intersects $f^{-k}(\cA'')$ in a
transversal way (see Figure~\ref{f.P_2}). 

Let $\Upsilon=f^{-1}(\Xi)=:\{(\theta,v(\theta))\ |\ \theta\in J\}$. Then, by
\eqref{e.curve_disjointness} above, $\Upsilon\subseteq J\times E$. Moreover, as
$v(\theta)\in E$ for all $\theta\in J$, we have
\begin{eqnarray*}
|\partial_\theta v(\theta)|&=&|(\partial_x
f_{\theta+\omega}^{-1})(\xi(\theta+\omega))\cdot \partial_\theta
\xi(\theta+\omega) +(\partial_\theta
f_{\theta+\omega}^{-1})(\xi(\theta+\omega))|\\
&=&\left|\frac{1}{(\partial_x
f_\theta)(v(\theta))}\cdot\partial_\theta
\xi(\theta+\omega)-\frac{(\partial_\theta
f_\theta)(v(\theta))}{(\partial_x f_\theta)(v(\theta))}\right|\\
&\leq& \alpha^{-2/p}\cdot S+\alpha^{-2/p}\cdot
S\stackrel{(\alpha)}{\leq} S.
\end{eqnarray*}

   \begin{figure}[ht!]
  \begin{center}
 \includegraphics[height=0.7\linewidth]{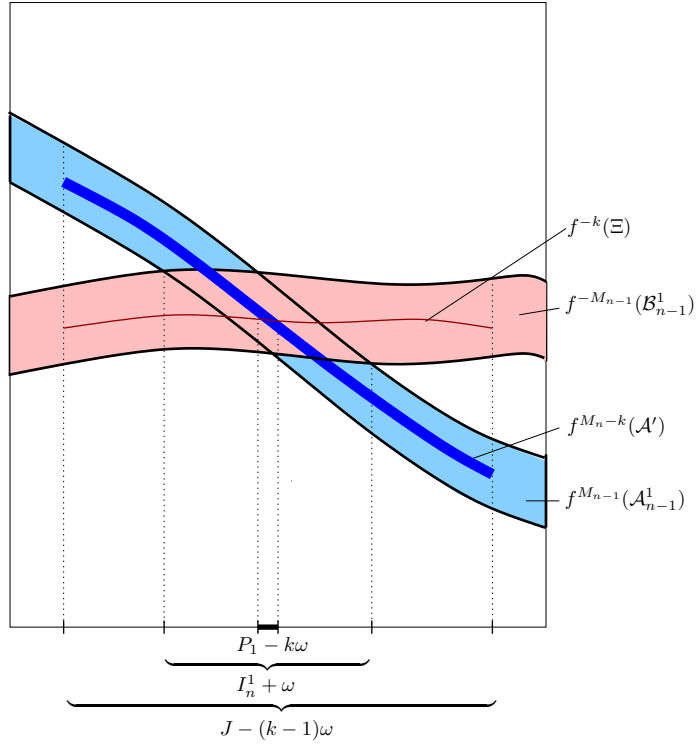}
  \end{center}

\caption{\small Transversal intersection between $f^{-k}(\Xi)$ and $f^{-k}(\cA'')=f^{M_n-k}(\cA')$.
  \label{f.P_2}}

\end{figure}
Let $\cA''':=f^{M_n-k}(\cA')$ as before and
$\Psi(\theta):=f_{\theta+(k-1)\omega}^{-(k-1)}(\upsilon(\theta+(k-1)\omega))$,
where $\theta\in J-(k-1)\omega$. Since $J\subseteq I_{n-1}^2\subseteq
I_{n-2}^2\subseteq \ldots \subseteq I_0^2$ by (\ref{claimresult-2}) and
(\ref{equ-c-4-1}), condition $(\mathcal{X})_{n-1}$ yields
\begin{equation}\label{equ-2}
 J\cap \cV_{n-1}^+=\emptyset.
 \end{equation}
 Further, as $J-(k-M_{n-1}-1)\omega\ssq I_{n-1}^1+(M_{n-1}+1)\omega$ by
 (\ref{claimresult-2}), (\ref{equ-c-4-1}), conditions $(\cX)_{n-1}$ and
 $(\cY'')_{n-1}$ imply $(J-(k-M_{n-1}-1)\omega)\cap
 \cW^-_{n-1}=\emptyset$. Together with (\ref{est_disjoint}), Lemma
 \ref{lem-ess_back} yields
\[
f^{-(k-M_{n-1}-1)}(\Upsilon)\ \subseteq \ \big(J-(k-M_{n-1}-1)\omega\big)\times
E \ \ssq \cB^1_{n-1} \ .
\] Thus,  we obtain
\begin{equation}\label{equ-3}
f^{-(k-1)}(\Upsilon)\subseteq f^{-M_{n-1}}(\cB_{n-1}^1) \subseteq
(I_{n-1}^1+\omega)\times E.
\end{equation}
Moreover, Lemma~\ref{l.interval_disjoint} yields that $(J-l\omega)\cap
(\cI_n+\omega)=\emptyset$ for $l=0,\ldots, k-2$. Therefore (\ref{equ-2}) and
the fact that $J-(k-1)\omega\ssq I_{n-1}^1+\omega$ by (\ref{claimresult-2}) and
(\ref{equ-c-4-1}) allow to apply Lemma \ref{estimate.back} in order to obtain
$$\sup_{\theta\in J-(k-1)\omega}|\partial_\theta\Psi|\ \leq\ \frac{S}{\alpha^{1/p}-1}.$$
Then by (\ref{esti-A'''-2}), we get $\inf_{\theta}
|\partial_\theta\varphi_{\cA'''}^\pm(\theta)|\stackrel{(\alpha)}{>} \sup_\theta
|\partial_\theta\Psi|$. Thus, by the same argument as in Section~\ref{s.P_0},
$\cA'''$ is above $f^{-(k-1)}(\Upsilon)$ on the left end of $J-(k-1)\omega$ and
below on the right end by (\ref{equ-3}) (see Figure \ref{f.P_1}). Therefore the
boundary curves of $\cA'''$ intersect $f^{-(k-1)}(\Upsilon)$ exactly once, which
means $\pi_1\left(\cA'''\cap f^{-(k-1)}(\Upsilon)\right)$ is an interval. Hence
$P_1=\pi_1(\cA'' \cap \Xi)=\pi_1\left(\cA'''\cap
  f^{-(k-1)}(\Upsilon)\right)+k\omega$ is an interval.
\qed\medskip

\subsection{Proof of Lemma \ref{l.extremal_positions} }

We will first prove that $(\cX)_n$ actually holds for $\tau=\tau^-,
\tau^+$. By Lemma \ref{l.perturbation_closereturn} and Proposition
\ref{finite_times}, we have $|I_n^\iota(\tau^\pm)|\leq
\varepsilon_n,\ \iota=1,2$, and $d(I_n^{1}(\tau^\pm)+k\omega,
I_n^2(\tau^\pm))=4\varepsilon_n$. If $d_H$ denotes the Hausdorff
distance, then $d_H(I_n^1(\tau^\pm)+k\omega, I_n^2(\tau^\pm))\leq
5\varepsilon_n$. Include $\tau^\pm$ throughout the proof.   The
Diophantine condition $\omega\in \cD(\gamma,\nu)$ implies
$d(I_n^\iota, I_n^\iota+j\omega)> 8\varepsilon_n$ for $\iota=1,2,\ j
\in [1, (2K_n+1)M_n] $ by (\ref{e.M_j}) and (\ref{e.eps_j}),
provided $\alpha$ is large. Then, given $l\in [1, 2K_nM_n]\smin\{k\}$,
we have
\begin{eqnarray*}
d(I_n^2, I_n^1+l\omega)\geq d(I_n^1+l\omega, I_n^1+k\omega)-
d_H(I_n^1+k\omega, I_n^2)
\ > \ 8\varepsilon_n-5\varepsilon_n=3\varepsilon_n,\ 
\end{eqnarray*}
and similarly
\begin{eqnarray*}
d(I_n^1, I_n^2+l\omega)\ \geq  \ 3\varepsilon_n \ .
\end{eqnarray*}
Moreover, by the choice of $\tau^-,\tau^+$ in Section~\ref{CloseReturn}, we have
$d(I_n^2, I_n^1+k\omega)>3\varepsilon_n$. Thus, $(\cX)_n$ is satisfied.

Hence, Proposition \ref{finite_times} implies that the two components of
$\cI_{n+1}$ are non-empty, which means $f^{M_n}(\cA_n^2)\cap
f^{-M_n}(\cB_n^2)\neq\emptyset$. Then Lemma \ref{l.inclusions} implies that
$\cB''\cap f^{M_n}(\cA_n^2)\neq \emptyset$. Moreover, Lemma \ref{l.P_0} implies
that $P_0\cap (I_n^1+(k+1)\omega)\neq \emptyset$. Since $|P_0|\leq
2\varepsilon_n$ (using the estimate from Lemma~\ref{l.P_0} and (\ref{e.eps_j})),
and $d(I_n^1+k\omega, I_n^2)=4\varepsilon_n$, we get $P_0 \cap
(I_n^2+\omega)=\emptyset$. When $\tau=\tau^-$, then since $I_n^2$ is to the
right of $I_n^1+k\omega$, we have $f^{M_n}(\cA_n^2)\subseteq \cR$, which means
$\cB''$ intersects $\cR$. Conversely, when $\tau=\tau^+$ we have
$f^{M_n}(\cA_n^2)\subseteq \cL$ since $I_n^2$ is to the left of $I_n^1+k\omega$
and thus $\cB''$ intersects $\cL$.\qed


\end{document}